\documentclass[11pt]{amsart}
\usepackage[english]{babel}
\usepackage{amsfonts,latexsym,amsthm,amssymb,graphicx}
\usepackage{mathabx}
\usepackage[all]{xy}
\usepackage[usenames]{color}
\usepackage{tikz}
\usetikzlibrary{shapes}
\usepackage{xypic}
\usetikzlibrary{decorations.pathreplacing}
\usepackage{amsmath}
\usepackage{graphicx}
\usepackage[enableskew]{youngtab}
\usepackage[utf8]{inputenc}
\usepackage[english]{babel}
\usepackage{tikz}
\usetikzlibrary{arrows}
\usepackage{float}
\usepackage{amscd}
\usepackage{tableau}
\usepackage{geometry}
\usepackage{amsrefs}
\usepackage{pdfsync}
\usepackage[colorlinks=true, linkcolor=blue, citecolor=blue, urlcolor=blue]{hyperref}

\DeclareMathOperator{\Gr}{Gr}
\DeclareMathOperator{\SL}{SL}
\newcommand{\C}{{\mathbb C}}
\newcommand{\Z}{{\mathbb Z}}

\newcommand{\cO}{{\mathcal O}}

\newcommand{\wtE}{\widetilde{E}}
\newcommand{\wto}{\widetilde{\omega}}

\DeclareMathOperator{\Fl}{Fl}
\newcommand{\IG}{\mathrm{IG}}
\newcommand{\QH}{\mathrm{QH}}

\DeclareMathOperator{\Sp}{Sp}
\newcommand{\IF}{\mathrm{IF}}
\DeclareMathOperator{\e}{\mathbf{e}}

\newcommand{\Mb}{\overline{\mathcal{M}}}
\DeclareMathOperator{\Span}{Span}

\DeclareMathOperator{\ev}{ev}

\DeclareFontFamily{OT1}{rsfs}{}
\DeclareFontShape{OT1}{rsfs}{n}{it}{<-> rsfs10}{}
\DeclareMathAlphabet{\mathscr}{OT1}{rsfs}{n}{it}

\CompileMatrices

\newtheorem{thm}{Theorem}[section]
\newtheorem{lemma}[thm]{Lemma}
\newtheorem{cor}[thm]{Corollary}
\newtheorem{prop}[thm]{Proposition}

{\theoremstyle{defn} \newtheorem{defn}[thm]{Definition}}
{\theoremstyle{remark} \newtheorem{remark}[thm]{Remark}
\newtheorem{example}[thm]{Example}}

\begin{document}

\title[Equivariant Quantum Cohomology of the Odd Symplectic Grassmannian]{
Equivariant Quantum Cohomology of the Odd Symplectic Grassmannian
}
\author{Leonardo C.~Mihalcea}
\author{Ryan M. Shifler}
\address{
Department of Mathematics,
Virginia Tech,
Blacksburg, VA 24061
}
\email{lmihalce@math.vt.edu}
\address{
Department of Mathematics,
Virginia Tech,
Blacksburg, VA 24061
}
\email{twigg@vt.edu}
\thanks{L.M. was supported in part by NSA Young Investigator Award 98320-16-1-0013 and a Simons Collaboration grant.}
\subjclass[2010]{Primary 14N35; Secondary 14N15, 14M15}

\begin{abstract}
The odd symplectic Grassmannian $\IG:=\IG(k, 2n+1)$ parametrizes $k$ dimensional subspaces of $\mathbb{C}^{2n+1}$ which are isotropic with respect to a general (necessarily degenerate) symplectic form. The odd symplectic group acts on $\IG$ with two orbits, and $\IG$ is itself a smooth Schubert variety in the submaximal isotropic Grassmannian $\IG(k, 2n+2)$. We use the technique of curve neighborhoods to prove a Chevalley formula in the equivariant quantum cohomology of IG,~i.e.~a formula to multiply a Schubert class by the Schubert divisor class. This generalizes a formula of Pech in the case $k=2$, and it gives an algorithm to calculate any multiplication in the equivariant quantum cohomology ring.
\end{abstract}

\maketitle


%
%

\section{Introduction}\label{intro}
Let $E:= \C^{2n+1}$ be an odd-dimensional complex vector space and $1 \le k \le n+1$. An {\em odd-symplectic form} $\omega$ on $E$ is a skew-symmetric bilinear form with kernel of dimension $1$. The {\em odd-symplectic Grassmannian} $\IG:= \IG(k, E)$ parametrizes $k$-dimensional linear subspaces of $E$ which are isotropic with respect to $\omega$. One can find vector spaces $F \subset E \subset \widetilde{E}$ such that $\dim F = 2n$, $\dim \widetilde{E} = 2n+2$, the restriction of $\omega$ to $F$ is non-degenerate, and $\omega$ extends to a symplectic form (hence non-degenerate) on $\widetilde{E}$. Then the odd-symplectic Grassmannian is an intermediate space
\begin{equation}\label{E:evenodd} \IG(k-1, F) \subset \IG(k, E) \subset \IG(k, \widetilde{E}) \/, \end{equation} sandwiched between two symplectic Grassmannians. This and the more general ``odd-symplectic partial flag varieties" have been studied by Mihai \cite{mihai:odd} and Pech \cite{pech:quantum}. In particular, Mihai showed that $\IG(k, E)$ is a smooth Schubert variety in $\IG(k, \widetilde{E})$, and that it admits an action of Proctor's {\em odd-symplectic group} $\Sp_{2n+1}$ (see \cite{proctor:oddsymgrps}). If $k \neq n+1$ then the odd-symplectic group acts on $\IG(k, E)$ with $2$ orbits, and the closed orbit can be identified with $\IG(k-1, F)$. If $k=1$ then $\IG(1, E) = \mathbb{P}(E)$ and if $k=n+1$ then $\IG(n+1, E)$ is isomorphic to the Lagrangian Grassmannian $\IG(n, F)$.

In this paper we are concerned with the study of the quantum cohomology ring $\QH^*_T(\IG)$ of the odd-symplectic Grassmannians, and its $T$-equivariant version, where $T$ is the maximal torus in $\Sp_{2n+1}$. Since $\IG$ is a Schubert variety in the symplectic Grassmannian $\IG(k, \wtE )$ it follows that the (equivariant) fundamental classes of those Schubert varieties $X(u) \subset \IG(k, \wtE )$ included in $\IG$ form a basis for the (co)homology ring $H^*(\IG)$; we call this the {\em Schubert basis}. This implies that the graded algebra $\QH^*_T(\IG)$ has a Schubert basis $[X(u)]_T$ over $H^*_T(pt)[q]$, indexed by a particular subset of Schubert classes in the quantum cohomology of $\QH^*_T(\IG(k, \widetilde{E}))$. There are Schubert multiplication formulas in $\QH^*_T(\IG)$, \[ [X(u)]_T \star [X(v)]_T = \sum_{w, d} c_{u,v}^{w,d} q^d [X(w)]_T  \/, \] where $c_{u,v}^{w,d}$ are the (equivariant) Gromov-Witten (GW) invariants for rational curves of degree $d$ in $\IG$, and $q$ is the quantum parameter. 

Our main result is a combinatorial formula for the multiplication $[X(u)]_T \star [D]_T$ of any Schubert class by the Schubert divisor class $[D]_T$. This is called the (equivariant, quantum) {\em Chevalley formula}. Before stating this formula explicitly, we discuss its significance. 

It has been known at least since Knutson and Tao's famous paper \cite{knutson.tao} that, despite the fact that the Schubert divisor does not generate the cohomology ring, the {\em equivariant} Chevalley formula gives a triangular system of equations calculating the Schubert structure constants.\begin{footnote}{This can be explained by using that the {\em localized} equivariant cohomology is generated by divisors; see \cite[\S 5]{BCMP:eqkt}.}\end{footnote} Knutson and Tao worked in the geometric context of the (ordinary) equivariant cohomology of the Grassmannian, and they used previous work of Okounkov-Olshanski \cite{okounkov,okounkov.olshanski} and Molev-Sagan \cite{molev.sagan} who studied a certain deformation of Schur polynomials, called factorial Schur functions. This system was extended to the equivariant quantum cohomology ring of flag manifolds in \cite{mihalcea:eqschub,mihalcea:eqqhom} and very recently to quantum K theory \cite{BCMP:eqkt}, but in these cases it is no longer triangular. In this paper we further extend these results to the case of odd-symplectic Grassmannians: we use the Chevalley  formula to obtain an algorithm for calculating any structure constant $c_{u,v}^{w,d}$; see theorem \ref{thm:alg} below. Its corollary \ref{cor:alg}, stated in the next section, makes precise the sense in which this formula determines the ring structure.  

The Chevalley formula technique was previously used to solve several cases of the {\em Giambelli problem}: find a presentation of the (equivariant, quantum) cohomology ring by generators and relations, then identify the polynomials which represent Schubert classes; see~e.~g.~\cite{tamvakis:giambelli} for more on the history of this problem.~Such results were obtained for the equivariant quantum cohomology ring of the Grassmannian \cite{mihalcea:giambelli}, of the orthogonal and Lagrangian Grassmannians \cite{IMN:factorial}, and of the equivariant cohomology of non-maximal isotropic Grassmannians \cite{tamvakis.wilson}. Although we do not pursue this application in this note, we believe that the Chevalley technique will be a key ingredient. A third application, for which we dedicate the upcoming paper \cite{li.mihalcea.shifler:O} joint with C.~Li, is to verify Galkin, Golyshev and Iritani's Conjecture $\mathcal{O}$ \cite{GGI,cheong.li:O} for the odd-symplectic Grassmannians. This uses the explicit combinatorial formulation of the Chevalley formula. 

\subsection{Statement of results} To state the Chevalley formula, we need to introduce a variant of $k$-strict partitions of Buch, Kresch and Tamvakis \cite{BKT2}. This variant was used by Pech in \cite{pech:thesis} to study the ordinary cohomology ring of $\IG$, and the quantum cohomology ring of $\IG(2,E)$ in \cite{pech:quantum}. 

A partition $(\lambda_1 \geq \lambda_2 \geq \cdots \geq \lambda_k)$ is $(n-k)$-strict if $\lambda_j>n-k$ implies $\lambda_j >\lambda_{j+1}$. Let $\Lambda$ be the set of $(n-k)$-strict partitions $(2n+1-k \geq \lambda_1 \geq \cdots \geq \lambda_k \geq -1)$ such that if $\lambda_k=-1$ then $\lambda_1=2n+1-k$. For each $\lambda \in \Lambda$ there is a Schubert variety in $X(\lambda) \subset \IG$ of codimension $|\lambda|:= \lambda_1 + \ldots + \lambda_k$. If $\lambda_1=2n+1-k$ and $\lambda_k \geq 0$ then let \[\lambda^{*}=(\lambda_2 \geq \lambda_3 \geq \cdots \geq \lambda_k \geq 0).\] If $\lambda_1<2n+1-k$ or $\lambda_k=-1$ then $\lambda^{*}$ does not exist. If $\lambda_1=2n+1-k$ and $\lambda_2=2n-k$ then let \[\lambda^{**}=(\lambda_1 \geq \lambda_3 \geq \cdots \geq \lambda_k \geq -1).\] If $\lambda_2<2n-k$ then $\lambda^{**}$ does not exist. With this notation, the Schubert divisor is $D= X(1)$. 
\begin{thm}[Quantum Chevalley formula]\label{Chevalley} Let $\lambda \in \Lambda$. Then the following holds in the equivariant quantum cohomology ring $\QH^*_T(\IG(k, \C^{2n+1})$: 
\begin{eqnarray}\label{E:eqchevintro}
[X(1)]_T \star [X(\lambda)]_T=\mbox{Classical Part}+q[X(\lambda^*)]_T+q[X(\lambda^{**})]_T \/.
\end{eqnarray}
The terms involving $\lambda^*$ or $\lambda^{**}$ are omitted if the corresponding partitions do not exist. The classical part consists of terms which do not involve $q$, and it is combinatorially explicit; see Theorem \ref{thm:eqqchev} below.
\end{thm}
This recovers Pech's results in \cite{pech:quantum} for the non-equivariant ring $\QH^*(\IG(2,E))$ and it verifies a conjecture for $\QH^*(\IG(3,E))$ stated in \cite{pech:thesis}. It is an easy exercise to check that for $k=1$ one obtains the quantum Chevalley formula in $\QH^*_T(\mathbb{P}^{2n})$ and that for $k= n+1$ it recovers the Chevalley formula for the Lagrangian Grassmannian $\IG(n ,F)$ from \cite{KT:qlagr,BKT}. As we mentioned above, the Chevalley multiplication yields an algorithm to calculate {\em any} equivariant GW invariant $c_{u,v}^{w,d}$, see \S \ref{s:alg} below. Its immediate corollary is:

\begin{cor}\label{cor:alg} Let $(A, +, \circ)$ be a graded, commutative, $H^*_T(pt)[q]$-algebra, with a $H^*_T(pt)[q]$-basis $ \{ a_\lambda \}_{\lambda \in \Lambda}$. Assume that the grading is the same as the one for the equivariant quantum ring $\QH^*_T(\IG)$, and assume that the Chevalley rule (\ref{E:eqchevintro}) holds in the basis $\{ a_\lambda \}$. Then the isomorphism of $H^*_T(pt)$-modules $A \to \QH^*_T(\IG(k, 2n+1))$ sending $a_\lambda \mapsto [X(\lambda)]_T$ is an isomorphism of algebras.\end{cor}

The proof of Theorem \ref{Chevalley} is based on the analysis of curve neighborhoods of Schubert varieties \cite{BCMP:qkfin,buch.m:nbhds}. Let $d \in H_2(\IG)$ be an effective degree and let $M_d:= \Mb_{0,2}(\IG, d)$ be the Kontsevich moduli space of stable maps \cite{FP} equipped with evaluation maps $\ev_1, \ev_2: M_d \to \IG$. To a closed subvariety $\Omega \subset \IG$  one can associate its {\em Gromov-Witten variety} $GW_d(\Omega):= \ev_1^{-1}(\Omega)$ and its {\em curve neighborhood} \[ \Gamma_d(\Omega) := \ev_2(GW_d(\Omega)) \subset \IG \/; \] see \S \ref{s:curvenbhds} below. The notion of curve neighborhoods is closely related to quantum cohomology. Roughly, let $X(\lambda) \subset \IG$ be a Schubert variety, and let $\Gamma_d(X(\lambda)) = \Gamma_1 \cup \Gamma_2 \cup \ldots \cup \Gamma_k$ be the decomposition of the curve neighborhood into irreducible components. By the divisor axiom, any component $\Gamma_i$ of ``expected dimension" will contribute to the quantum product $[X(1)]_T \star [X(\lambda)]_T$ with $ d \cdot m_i \cdot q^d [\Gamma_i]_T$, where $m_i$ is the degree of $\ev_2: GW_d(X(u)) \to \Gamma_d(X(\lambda))$ over the given component. Therefore the main task is to find the components $\Gamma_i$ of expected dimension, and calculate the associated multiplicities $m_i$. It was proved in \cite{buch.m:nbhds} that for homogeneous spaces the curve neighborhoods of Schubert varieties are irreducible and that all the multiplicities $m_i =1$. But $\IG$ is no longer homogeneous, and one easily finds reducible curve neighborhoods. Nevertheless, we proved that for any Schubert variety $X(\lambda)$, the only curve neighborhoods $\Gamma_d(X(\lambda))$ of expected dimension are those when $d=1$ and $X(\lambda)$ is included in the closed $\Sp_{2n+1}$-orbit of $\IG$. In this case \[ \Gamma_1(X(\lambda)) = X(\lambda^1) \cup X(\lambda^{2}) \/, \] and the components of the expected dimension correspond respectively to the partitions $\lambda^*, \lambda^{**}$ from Theorem \ref{Chevalley} above. (But it may happen, for instance, that $X(\lambda^1)$ does not have expected dimension, in which case $\lambda^*$ does not exist.) We refer to Theorems \ref{thm:vanishing} and \ref{thm:curveNBHDS} for precise statements. Furthermore, the morphism $\ev_2: GW_1(X(\lambda)) \to \Gamma_1(X(\lambda))$ is birational over the relevant components, thus all multiplicities $m_i =1$. This is achieved in sections \ref{s:lines} and \ref{s:linenbhds} by using a rather delicate analysis of the space of lines in $\IG$. In the course of the proof we showed that each orbit of $\Sp_{2n+1}$ in $\IG$ contributes with at most one component to the curve neighborhood. It is tempting to conjecture that this pattern extends at least to the odd-symplectic partial flag varieties.  

Besides reducibility of curve neighborhoods of Schubert varieties, it is also worth pointing out another 
difference between the (quantum) cohomology of odd-symplectic Grassmannians and that of flag manifolds. Many arguments in the Gromov-Witten theory of homogeneous spaces rely on the Kleiman-Bertini transversality theorem, which makes the GW invariants enumerative. Variants of this theorem exist for varieties with a group acting with finitely many orbits (see e.g. \cite{graber}). But the lack of a transitive group action implies that occasionally cycles in $\IG$ cannot be translated to general position, and that it is possible that certain Schubert multiplications might be non-effective. Indeed, Pech found Pieri-type multiplications, both in ordinary and in quantum cohomology of $\IG(2,E)$, which yield negative structure constants $c_{u,v}^{w,d}$. For more such examples, equivariant or not, {\color{black} see the table for $\QH^*_T(\IG(2,\C^5))$ in \S \ref{s:examples}, and the remark \ref{rmk:pos} below.} We partially circumvented the non-transversality problem by employing the aforementioned technique of curve neighborhoods.

\subsection{Organization of the paper} The sections 2-5 are dedicated to recalling the basic definitions and the relevant facts on (equivariant) quantum cohomology. In section \ref{s:curvenbhds} we discuss curve neighborhoods of Schubert varieties. The main result is Theorem \ref{Lem10} about estimates on the dimension of curve neighborhoods. This is then used in section \ref{s:vanishing} to prove the vanishing of many GW invariants. In sections \ref{s:lines} we study the moduli space of lines on $\IG$; our main result is Corollary \ref{cor:GW} where we prove that if $\Omega$ is a Schubert variety included in the closed orbit then the GW variety $GW_1(\Omega)$, has two irreducible, generically reduced components. This result is used in Section \ref{s:linenbhds} to obtain similar results about the curve neighborhood $\Gamma_1(\Omega)$. In section \ref{s:gwlines} we prove birationality results and use this to calculate all the non-vanishing equivariant GW invariants which appear in the Chevalley formula; see Theorem \ref{thm:gw1} and Corollary \ref{cor:gw1}. In section \ref{s:partitions} we re-interpret all results in terms of $k$-strict partitions, and obtain the statement of Theorem \ref{Chevalley} above. The algorithm to calculate the full multiplication table in $\QH^*_T(\IG)$ is presented in section \ref{s:alg}. {\color{black} Section \ref{s:examples} includes examples of products in $\QH^*_T(\IG(2,\C^5))$ and $\QH^*_T(\IG(3,\C^7))$.}

{\em Acknowledgements.} We would like to thank Dan Orr and Mark Shimozono for discussions and valuable suggestions and to Pierre-Emmanuel Chaput, Changzheng Li, and Nicolas Perrin for discussions and collaborations on related projects. Special thanks are due to Anders Buch for encouragement and interest in this project.

\section{Preliminaries}

\subsection{The odd symplectic group} We recall next the definition and basic properties of odd symplectic flag manifolds, following Mihai's paper \cite{mihai:odd}; see also \cite{mihai:thesis, pech:quantum}. Let $E$ be a complex vector space of dimension $\dim_{\C} E = 2n+1$, and let $\omega$ be an odd-symplectic form on $E$, i.e. bilinear, skew-symmetric, with kernel of dimension $1$. The {\em odd symplectic group} is the subgroup of $GL(E)$ which preserves this symplectic form: \[ \Sp_{2n+1}(E) := \{ g \in GL(E): \omega(g.u, g.v) = \omega(u,v), \forall u,v \in E \} \/. \] It will be convenient to extend the form $\omega$ to a nondegenerate symplectic form $\widetilde{\omega}$ on an even dimensional space $\widetilde{E} \supset E$, and to identify $E \subset \widetilde{E}$ with a coordinate hyperplane $\C^{2n+1} \subset \C^{2n+2}$. For that, let $\{ \e_1,\ldots , \e_{n+1}, \e_{\overline{n+1}}, \ldots , \e_{\bar{2}}, \e_{\bar{1}} \}$ be the standard basis of $\widetilde{E}:=\mathbb{C}^{2n+2}$, where $\bar{i}=2n+3-i$. Set $|i|=\min\{i,\bar{i} \}$, and consider $\widetilde{\omega}$ to be the nondegenerate symplectic form on $\widetilde{E}$ defined by \[ \widetilde{\omega}(\e_i,\e_j)=\delta_{i,\bar{j}} \mbox{ for all } 1 \leq i\leq j \leq \bar{1} \/.\] The form $\widetilde{\omega}$ restricts to the degenerate symplectic form $\omega$ on $E:=\mathbb{C}^{2n+1}=\left<\e_1, \e_2,\cdots, \e_{2n+1} \right>$ such that the kernel $\ker \omega$ is generated by $\e_1$. Then \[\omega(\e_i,\e_j)=\delta_{i,\bar{j}} \mbox{ for all } 1 \leq i\leq j \leq \bar{2}. \] Let $F \subset E$ denote the $2n$ dimensional vector space with basis $\{ \e_2, \e_3,\cdots,\e_{2n+1} \}$. Since $F \cap \ker \omega = (0)$ it follows that $\omega$ restricts to a nondegenerate form on $F$. Let $\Sp_{2n}(F)$ and $\Sp_{2n+2}(\widetilde{E})$ denote the symplectic groups which preserve respectively the symplectic form $\omega_{|F}$ and $\wto$. Then with respect to the decomposition $E = F \oplus \ker \omega$ the elements of the odd-symplectic group $\Sp_{2n+1}(E)$ are matrices of the form \begin{eqnarray*}
\Sp_{2n+1}(E)=\left\{\left( \begin{array}{ccc}
\lambda & a  \\
0 & S  \\
 \end{array} \right): \lambda \in \mathbb{C}^*, a \in \mathbb{C}^{2n}, S \in \Sp_{2n}(F) \right\}.
\end{eqnarray*}
The symplectic group $\Sp_{2n}(F)$ embeds naturally into $\Sp_{2n+1}(E)$ by $\lambda = 1$ and $a = 0$, but $\Sp_{2n+1}(E)$ is {\em not} a subgroup of $\Sp_{2n+2}(\widetilde{E})$.\begin{footnote}{However, Gelfand and Zelevinsky \cite{gelfand.zelevinsky} defined  another group $\widetilde{\Sp}_{2n+1}$ closely related to $\Sp_{2n+1}$ such that $\Sp_{2n} \subset \widetilde{\Sp}_{2n+1} \subset \Sp_{2n+2}$.}\end{footnote}~Mihai showed in \cite[Prop. 3.3]{mihai:odd} that there is a surjection $P \to \Sp_{2n+1}(E)$ where $P \subset \Sp_{2n+2}(\wtE)$ is the parabolic subgroup which preserves $\ker \omega$, and the map is given by restricting $g \mapsto g_{|E}$. Then the Borel subgroup $B_{2n+2} \subset \Sp_{2n+2}(\wtE)$ of upper triangular matrices restricts to the (Borel) subgroup $B \subset \Sp_{2n+1}(E)$. Similarly, the maximal torus $T_{2n+2}:=\{ \mbox{diag}(t_1,\cdots,t_{n+1},t_{n+1}^{-1},\cdots, t_1^{-1}): t_1,\cdots,t_{n+1} \in \mathbb{C}^* \} \subset B_{2n+2}$ restricts to the maximal torus \[ T=\{ diag(t_1,\cdots,t_{n+1},t_{n+1}^{-1},\cdots, t_{2}^{-1}): t_1,\cdots,t_{n+1} \in \mathbb{C}^* \} \subset B \/. \]

\subsection{The odd symplectic flag varieties} Let $1 \le i_1 < \ldots < i_r \le n+1$. The {\em odd symplectic flag variety} $\IF(i_1, \ldots , i_r; E)$ consists of flags of linear subspaces $F_{i_1} \subset \ldots \subset F_{i_k} \subset E$ such that $\dim F_{i_j} = i_j$ and $F_{i_j}$ is isotropic with respect to the symplectic form $\omega$. The inclusion $E \subset \wtE$ makes it a closed subvariety of the (even) symplectic flag variety $\IF(i_1, \ldots , i_r; \wtE)$ which consists of similar flags of subspaces, isotropic with respect to the symplectic form $\wto$. The latter is a homogeneous space for $\Sp_{2n+2}(\wtE)$. In fact, the inclusions $F \subset E \subset \wtE$ realize the odd-symplectic flag variety as an intermediate variety between two consecutive symplectic flag varieties: \[ \IF(i_1 -1, \ldots , i_r -1; F) \subset \IF(i_1, \ldots , i_r; E) \subset \IF(i_1, \ldots, i_r; \wtE) \/, \] where the flags in $\IF(i_1 -1, \ldots , i_r -1; F)$ are isotropic with respect to $\omega_{|F}$. If $r=1$ then we obtain the sequence of inclusions of Grassmannians shown in equation (\ref{E:evenodd}). There is a natural embedding of the odd symplectic flag variety as a closed subvariety of the type A partial flag variety $\Fl(i_1, \ldots , i_r; E)$ which parametrizes flags of given dimensions in $E$. It turns out that $\IF(i_1, \ldots , i_r; E)$ is a smooth subvariety of $\Fl(i_1, \ldots, i_r; E)$ of codimension $\frac{i_r(i_r-1)}{2}$; see \cite[Prop. 4.1]{mihai:odd} for details. The odd-symplectic group acts on $\IF(i_1, \dots, i_r;E)$, but the action is no longer transitive. The next result, due to Mihai (see \cite[Propositions 5 and 6]{mihai:odd} describes the orbits of this action.
\begin{prop}\label{eee}
The odd symplectic group $\Sp_{2n+1}(E)$ acts on $\IF(i_1, \cdots, i_r; E)$ with $r+1$ orbits if  $i_r<n+1$ and $r$ orbits if $i_r=n+1$. The orbits are:
\begin{eqnarray*}
\cO_j&=&\{V_{i_1} \subset \cdots \subset V_{i_r} \subset E: e_1 \in V_{i_j}, e_1\notin V_{i_{j-1}} \} \mbox{ for all } 1 \leq j \leq r\\
\mbox{ and } \cO_{r+1}&=&\{V_{i_1} \subset \cdots \subset V_{i_r} \subset E:  e_1\notin V_{i_{r}} \} \mbox{ if } i_r<n+1,
\end{eqnarray*} where by convention $V_{i_0} = (0)$. The only closed orbit is $\cO_1$, and it may be naturally identified to $\IF(i_1 - 1, \ldots , i_r -1;F)$. 

In particular, for $1 \leq k \leq n$ the odd symplectic group $\Sp_{2n+1}(E)$ acts on the odd symplectic Grassmannian $\IG(k,E)$ with two orbits
\begin{eqnarray*}
X_c&=&\{V \in \IG(k,E) : e_1 \in V\} \mbox{~ the closed orbit}\\
X^{\circ}&=&\{V \in \IG(k,E) : e_1 \notin V\} \mbox{~the open orbit}.
\end{eqnarray*}
The closed orbit $X_c$ is isomorphic to $\IG(k-1,F)$. If $k=n+1$ then $\IG(n+1,E)= X_c$ may be identified to the Lagrangian Grassmannian $\IG(n,F)$. 
\end{prop}
Mihai identifies the closures $\overline{\cO}_i$ of the orbits and proves they are smooth. From now on we will identify $F \subset E \subset \widetilde{E}$ to $\C^{2n} \subset \C^{2n+1} \subset \C^{2n+2}$ with bases $\langle \e_2, \ldots , \e_{2n+1} \rangle \subset \langle \e_1 , \ldots , \e_{2n+1} \rangle \subset \langle \e_1, \ldots , \e_{2n+2} \rangle$. The corresponding symplectic flag manifolds will be denoted by $\IF(i_1-1, \ldots , i_r-1; 2n) \subset \IF(i_1, \ldots , i_r; 2n+1) \subset \IF(i_1, \ldots , i_r; 2n+2) $. Similarly $\Sp_{2n+1}(E)$ will be denoted by $\Sp_{2n+1}$ etc.

\subsection{The Weyl group and odd-symplectic minimal representatives} We recall next the indexing sets which we will use in the next section to define the Schubert varieties. 

Consider the root system of type $C_{n+1}$ with positive roots $R^+=\{t_i \pm t_j: 1 \leq  i < j \leq n+1\} \cup \{ 2 t_i : 1 \leq  i \leq n+1 \}$ and the subset of simple roots $\Delta=\{\alpha_i:= t_i-t_{i+1}: 1 \leq i \leq n \} \cup \{ \alpha_{n+1}:= 2t_{n+1} \}$. The associated Weyl group $W$ is the hyperoctahedral group consisting of {\em signed permutations}, i.e. permutations $w$ of the elements $\{1, \cdots, n+1,\overline{n+1},\cdots,\overline{1}\}$ satisfying $w(\overline{i})=\overline{w(i)}$ for all $w \in W$. For $1 \le i \le n$ denote by $s_i$ the simple reflection corresponding to the root $t_i - t_{i+1}$ and $s_{n+1}$ the simple reflection of $2 t_{n+1}$. Each subset $ I:=\{ i_1 < \ldots < i_r \} \subset \{ 1, \ldots , n+1 \} $ determines a parabolic subgroup $P:=P_I \le \Sp_{2n+2}(\wtE)$ with Weyl group $W_{P} = \langle s_i: i \neq i_j \rangle$ generated by reflections with indices {\em not} in $I$. Let $\Delta_P:= \{ \alpha_{i_s}: i_s \notin \{ i_1, \ldots , i_r \} \}$ and $R_P^+ := \Span_\Z \Delta_P \cap R^+$; these are the positive roots of $P$. Let $\ell:W \to \mathbb{N}$ be the length function and denote by $W^{P}$ the set of minimal length representatives of the cosets in $W/W_{P}$. The length function descends to $W/W_P$ by $\ell(u W_P) = \ell(u')$ where $u' \in W^P$ is the minimal length representative for the coset $u W_P$. We have a natural ordering \[ 1 < 2 < \ldots < n+1 < \overline{n+1} < \ldots < \overline{1} \/, \] which is consistent with our earlier notation $\overline{i} := 2n+3 - i$. Let $P=P_k$ to be the maximal parabolic obtained by excluding the reflection $s_k$. Then the minimal length representatives $W^{P}$ have the form $(w(1)<w(2)<\cdots<w(k)| w(k+1)< \ldots < w(n+1) \le n+1)$ if $k < n+1$ and $(w(1)<w(2)<\cdots<w(n+1))$ if $k= n+1$. Since the last $n+1-k$ are determined from the first, we will identify an element in $W^{P_k}$ with the sequence $(w(1)<w(2)<\cdots<w(k))$. 

\begin{example} 
\label{Ex}
The reflection $s_{t_1+t_2}$ is given by the signed permutation $s_{t_1+t_2}(1)=\bar{2}, s_{t_1+t_2}(2)=\bar{1}, \mbox{ and } s_{t_1+t_2}(i)=i$ for all $3 \leq i \leq n+1$. The minimal length representative of $s_{t_1+t_2}W_{P_k}$ is $(3<4<\cdots<k<\bar{2}<\bar{1})$.
\end{example}
The Weyl group $W$ admits a partial ordering $\leq$ given by the {\em Bruhat} order. Its covering relations are given by $w < ws_\alpha$ where $\alpha \in R^+$ is a root and $\ell(w) < \ell(w s_\alpha)$. We will use the {\it Hecke product} on the Weyl group $W$. For a simple reflection $s_i$ the product is defined by \[w \cdot s_i = \left\{
     \begin{array}{lr}
       ws_i &  \mbox{if } \ell(ws_i)>\ell(w); \\
       w &  \mbox{otherwise}
     \end{array} \right. \]
The Hecke product gives $W$ a structure of an associative monoid; see e.g.~\cite[\S 3]{buch.m:nbhds} for more details.~Given $u,v \in W$, the product $uv$ is called {\it reduced} if $\ell(uv)=\ell(u)+\ell(v)$, or, equivalently, if $uv = u \cdot v$. For any parabolic group $P$, the Hecke product determines a left action $W \times W/W_{P} \longrightarrow W/W_{P}$ defined by \[u \cdot (wW_{P})=(u \cdot w) W_P. \] We recall the following properties of the Hecke product (cf.~e.g. \cite{buch.m:nbhds}). 

\begin{lemma}\label{lemma:bound} For any $u,v \in W$ there is an inequality $\ell(u \cdot vW_{P}) \leq \ell(u)+\ell(vW_{P})$. If the equality holds then $u \cdot v W_{P}=uvW_{P}$. If furthermore $v \in W^P$ is a minimal length representative, then the following are equivalent:

(i) $\ell(u \cdot vW_{P}) = \ell(u)+\ell(vW_{P})$;

(ii) $\ell( u \cdot v) = \ell(u) + \ell(v)$, and $u \cdot v = uv$ is a minimal length representative in $W^P$. \end{lemma}
\begin{proof} The first part of this lemma is explicitly stated in \cite[\S 3]{buch.m:nbhds}. For the equivalence, observe first that (ii) implies (i) since $u \cdot v = uv$. For the converse, since $v \in W^P$, \[ \ell(u \cdot v) \ge \ell(u \cdot v W_P) = \ell(u) + \ell(vW_P) = \ell(u) + \ell(v) \ge \ell(u \cdot v)\/.\] Thus $u \cdot v \in W^P$ and $\ell(u \cdot v) = \ell(u) + \ell(v)$ and this finishes the proof.\end{proof}

\subsection{Schubert Varieties in even and odd flag manifolds} Let $ I:=\{ i_1 < \ldots < i_r \} \subset \{ 1, \ldots , n+1 \} $ and the associated parabolic subgroup $P:=P_I$. The even symplectic flag manifold $\IF(i_1, \ldots , i_r; 2n+2)$ is a homogeneous space $\Sp_{2n+2}/P$. For each $w \in W^{P}$ let $Y(w)^\circ:=B_{2n+2} w B_{2n+2}/P$ be the {\em Schubert cell}. This is isomorphic to the space $\C^{\ell(w)}$. Its closure $Y(w):=\overline{Y(w)^\circ}$ is the {\em Schubert variety}. We might occasionally use the notation $Y(w W_P)$ if we want to emphasize the corresponding coset, or if $w$ is not necessarily a minimal length representative. Recall that the Bruhat ordering can be equivalently described by $v \le w$ if and only if $Y(v) \subset Y(w)$. Set \begin{equation}\label{E:w0} w_0 = \begin{cases} ( \overline{2}, \overline{3}, \ldots , \overline{n+1}, 1) & \textrm{ if } k < n+1 \/; \\ ( 1, \overline{2}, \overline{3}, \ldots , \overline{n+1}) & \textrm{ if } k = n+1 \/; \end{cases} \end{equation} this is an element in $W$.
Let $B:= B_{2n+2} \cap \Sp_{2n+1}$ be the odd-symplectic Borel subgroup. The following results were proved by Mihai \cite[\S 4]{mihai:odd}. 

\begin{prop}\label{prop:schubert} (a) The natural embedding $\iota: \IF(i_1, \cdots, i_r; 2n+1) \hookrightarrow \IF(i_1, \cdots, i_r; 2n+2)$ identifies $\IF(i_1, \cdots, i_r; 2n+1)$ with the (smooth) Schubert subvariety \[ Y(w_0 W_P) \subset \IF(i_1, \cdots, i_r; 2n+2) \/.\] 

(b) The Schubert cells (i.e. the $B_{2n+2}$-orbits) in $Y(w_0)$ coincide with the $B$-orbits in $\IF(i_1, \ldots , i_r; 2n+1)$. In particular, the $B$-orbits in $\IF(i_1, \ldots , i_r; 2n+1)$ are given by the Schubert cells $Y(w)^\circ \subset \IF(i_1, \ldots , i_r; 2n+2)$ such that $w \le w_0$. 
\end{prop}
To emphasize that we discuss Schubert cells or varieties in the odd-symplectic case, for each $w \le w_0$ such that $w \in W^P$, we denote by $X(w)^\circ$, and $X(w)$, the Schubert cell, respectively the Schubert variety in $\IF(i_1, \ldots , i_r; 2n+1)$. The same Schubert variety $X(w)$, but regarded in the even flag manifold is denoted by $Y(w)$. For further use we note that $\IG(k, 2n+1)$ has complex codimension $k$ in $\IG(k, 2n+2)$. Further, a Schubert variety $X(w)$ in $\IG(k, 2n+1)$ is included in the closed orbit $X_c$ of if and only if it has a minimal length representative $w \le w_0$ such that $w(1)=1$. 

Define the set $W^{odd}:= \{ w \in W: w \le w_0 \}$ and call its elements {\em odd-symplectic permutations}. This set consists of permutations $w \in W$ such that $w(j) \neq \bar{1}$ for any $1 \le j \le n+1$. The following closure property of the Hecke product on odd-symplectic permutations will be important later on.
\begin{lemma}\label{lemma:hclosure} Let $u,v \in W$ be two odd-symplectic permutations, and assume that $u(1) =1$. Then $uv$ and $u \cdot v$ are odd-symplectic permutations.\end{lemma}
 
\begin{proof} We need to show that $(uv)(j) \neq \bar{1}$ and $(u \cdot v)(j) \neq \bar{1}$ for any $1 \le j \le n+1$. In the first situation, since $u(1) =1$, if $(uv)(j) =\bar{1}$ for some $1 \le j \le n+1$, then $v(j)= \bar{1}$, which contradicts that $v$ is odd-symplectic. For the second, consider the signed permutation $u':= (u \cdot v)v^{-1} \in W$. By \cite[Prop. 3.1]{buch.m:nbhds} we have that $u' \le u$ and $u' v = u \cdot v$. The condition that $u' \le u$ implies that there is an inclusion of Schubert varieties $X(u') \subset X(u)$ in the full odd-symplectic flag manifold $\IF^{odd}:= \IF(1,2,\ldots , n+1; 2n+1)$. Further, the hypothesis that $u(1) = 1$ implies that $X(u)$ is in the closed orbit of $\IF^{odd}$, thus $X(u')$ is in the closed orbit as well. This implies that $u'(1) = 1$. Then $u \cdot v = u' v$ and since $u'(1) =1$ the element $u'v$ is again odd-symplectic, as claimed.\end{proof}

\section{Equivariant cohomology} Fix a parabolic subgroup $P \subset \Sp_{2n+2}$ containing the standard Borel subgroup $B_{2n+2}$. Let $\IF^{ev}:= \IF(i_1, \ldots , i_r; 2n+2)$ be the corresponding symplectic flag variety. The Schubert cells $Y(w)^\circ$ form a stratification
of $\IF^{ev}$, when $w$ varies in $W^P$. This implies that the Schubert classes $[Y(w)] \in H_{2 \ell(w)}(\IF)$ form a basis of the (integral) homology of $\IF^{ev}$. Since $\IF^{ev}$ is smooth, the Schubert classes determine cohomology classes $[Y(w)] \in H^{2 \dim \IF - 2 \ell(w)}(\IF^{ev})$. The odd-symplectic flag manifold $\IF:= \IF(i_1, \ldots , i_r; 2n+1)$ is a smooth Schubert variety in $\IF^{ev}$, therefore its Schubert classes $[X(w)] = [Y(w)] \in H_{2 \ell(w)} (\IF)$ for $w \in W^P \cap W^{odd}$ form a $\Z$-basis for both homology and cohomology $H_*(\IF)= H^*(\IF)$. We will use that in the Grassmannian case, the inclusion $\iota: \IG(k, 2n+1) \to \IG(k, 2n+2)$ gives a group isomorphism $H_2(\IG(k, 2n+1)) \simeq H_2(\IG(k, 2n+2))$ sending the Schubert curve $[X(s_k)]$ to itself. For $\mathcal{X} \in \{ \IF, \IF^{ev} \}$ there is a nondegenerate Poincar{\'e} pairing $\langle \cdot , \cdot \rangle : H^*(\mathcal{X}) \otimes H^*(\mathcal{X}) \to H^*(pt)$ given by \[ \langle \gamma_1, \gamma_2 \rangle = \int_{\mathcal{X}} \gamma_1 \cup \gamma_2 \] where the integral is the push-forward to the point, i.e. $\int_{\mathcal{X}} \gamma:= p_*( \gamma)$ and $p:\mathcal{X} \to pt$ is the structure morphism. For a cohomology class $\gamma \in H^*(\mathcal{X})$ we denote by $\gamma^\vee$ its Poincar{\'e} dual. Thus $\int_{\IF} [X(u)] \cup [X(v)]^\vee = \delta_{u,v}$.   

\begin{remark} It is well known that the Poincar{\'e} dual of a Schubert class in $H^*(\IF^{ev})$ is again a Schubert class; indeed this is true for any homogeneous space $G/P$ \cite{brion:lectures}. This is no longer true in the odd-symplectic case. Formulas for Poincar{\'e} dual classes of Schubert classes in the odd-symplectic Grassmannian $\IG(k,2n+1)$ were calculated by Pech in \cite[Prop. 3]{pech:quantum} for $k=2$ and in \cite[Prop. 2.11, p.50]{pech:thesis} for arbitrary $k$.\end{remark}

We review some basic facts about the equivariant cohomology ring, following \cite{anderson:notes}, and focusing on $H^*_T(\IF)$. For any topological space $Z$ with a left torus $T$ action, its {\em equivariant cohomology} ring is the ordinary cohomology of the Borel mixed space $Z_T:= (ET \times Z)/T$ where $ET \to BT$ is the universal $T$-bundle, and $T$ acts on $ET \times Z$ by $t \cdot (e, z) = (e t, t^{-1}z)$. In particular, $H^*_T(pt) = H^*(BT)$ is a polynomial ring $\Z[t_1, \ldots , t_s]$ where $t_i$ are an additive basis for $(Lie~ T)^*$. The continuous map $Z_T \to BT$ gives a $H^*_T(pt)$-algebra structure on $H^*_T(Z)$. Let now $Z= \IF$ with its natural $T \simeq (\mathbb{C}^*)^{n+1}$ action. The Schubert varieties $X(u) \subset \IF$ are $T$-stable, and the fundamental classes $[X(u)]_T \in H_{2 \ell(u)}^T(\IF)$ give an $H^*_T(pt)$-basis for the equivariant (co)homology $H^*_T(\IF) = H_*^T(\IF)$. The inclusion $\iota: \IF \to \IF^{ev}$ gives a natural restriction map $H^*_{T_{2n+2}}(\IF^{ev}) \to H^*_T(\IF)$. The action of $T_{2n+2}$ on $\IF$ factors through that of $T$, therefore the natural morphism $H^*_{T_{2n+2}}(\IF) \to H^*_T(\IF)$ is an algebra isomorphism over $H^*_T(pt)= H^*_{T_{2n+2}}(pt)= \Z[t_1, \ldots t_{n+1}]$. Because of this, we will take $T:=T_{2n+2}$ from now on.
We use the same conventions as in \cite[\S 8]{buch.m:nbhds} for the geometric interpretation of the characters $t_i$ inside the equivariant cohomology ring. There is an equivariant version of the Poincar{\'e} pairing $\langle \cdot , \cdot \rangle :H^*_T(\IF) \otimes H^*_T(\IF) \to H^*_T(pt)$ given by the (equivariant) push forward map to the point: \[ \langle \gamma_1 , \gamma_2 \rangle = \int^T_{\IF^{odd}} \gamma_1 \cup \gamma_2: = p_*^T(\gamma_1 \cup \gamma_2) \in H^*_T(pt) \/. \]

\section{(Equivariant) Quantum cohomology} In this section we recall some basic facts about equivariant Gromov-Witten (GW) invariants and the equivariant quantum (EQ) cohomology rings, following \cite{FP,mihalcea:eqqhom}. For the purposes of this paper we specialize to the odd and even-symplectic Grassmannian case.

\subsection{Equivariant Gromov-Witten invariants} Set $\IG:= \IG(k, 2n+1)$ and $\IG^{ev}:= \IG(k, 2n+2)$, with $\iota: \IG \to \IG^{ev}$ the natural embedding. Let $\mathcal{X} \in \{ \IG, \IG^{ev} \}$. Recall that $H_2(\mathcal{X}) = \mathbb{Z}$. A {\em degree} $d$ in $\mathcal{X}$ is an effective homology class $d \in H_2(\mathcal{X})$, and it can be identified with a non-negative integer. Let $\Mb_{0,r}(\mathcal{X}, d)$ be the Kontsevich moduli of stable maps to $\mathcal{X}$ of degree $d$ to $\mathcal{X}$ with $r$ marked points ($r \geq 0$); see e.g. \cite{FP}. This is a projective algebraic variety of expected dimension \[ \textrm{expdim}~ \Mb_{0,r}(\mathcal{X},d) = \dim \mathcal{X} + \int_{[X(s_k)]} c_1(T_{\mathcal{X}}) + r-3 \] where $T_{\mathcal{X}}$ denotes the tangent bundle of $\mathcal{X}$. 

\begin{lemma}\label{lemma:degq} Let $X(Div)$ and $Y(Div)$ be the (unique) Schubert divisors in $\IG$ and $\IG^{ev}$. Then the following equalities hold:  

(a) $\iota^*[Y(Div)] = [X(Div)]$;

(b) $c_1(T_{\IG}) = (2n+2 - k) [X(Div)]$ and $c_1(T_{\IG^{ev}}) = (2n+3 - k) [Y(Div)]$.

(c) \[ \int_{[X(s_k)]} c_1(T_{\mathcal{X}}) = \begin{cases} 2n+2- k & \textrm{ if } \mathcal{X} = \IG;\\ 2n+3- k & \textrm{ if } \mathcal{X} = \IG^{ev} \/. \end{cases} \]

\end{lemma}

\begin{proof} A more general version of the first identity was proved by Pech in her thesis \cite[Prop. 2.9]{pech:thesis}. The explicit calculation of the class of the tangent bundle in the even case can be found e.g. in \cite{BKT2}. In the odd case, the calculation is implicit Pech's work (cf. \cite[Prop. 13]{pech:quantum}, see also \cite[Prop. 2.15]{pech:thesis}). Part (c) is a standard calculation based on the fact that $\int_{\IG^{ev}} [Y(Div)] \cap [Y(s_k)] =1$.\end{proof}

The points of the moduli space are (equivalence classes of) stable maps $f: (C, pt_1, \ldots  , pt_r) \to \mathcal{X}$ of degree $d$, where $C$ is a tree of $\mathbb{P}^1$'s and $pt_i \in C$ are non-singular points. The 
moduli space $\Mb_{0,r}(\mathcal{X},d)$ comes equipped with $r$ evaluation maps $\ev_i: \Mb_{0,r}(\mathcal{X},d) \to \mathcal{X}$ sending $(C, pt_1, \ldots  pt_r;f)$ to $f(pt_i)$. For $\gamma_1, \ldots ,  \gamma_r \in H^*_T(\mathcal{X})$, the $r$-point, genus $0$, (equivariant) GW invariant is defined by 
\[ \langle \gamma_1, \ldots , \gamma_r \rangle_d := \int_{[\Mb_{0,r}(\mathcal{X},d)]^{vir}}^T \ev_1^*(\gamma_1) \cup \ev_2^*(\gamma_2) \cup \ldots \cup \ev_r^*(\gamma_r) \in H^*_T(pt) \/, \] where $[\Mb_{0,r}(\mathcal{X},d)]_T^{vir} \in H^T_{2 \textrm{ expdim } \Mb_{0,r}(\mathcal{X},d)}(\Mb_{0,r}(\mathcal{X},d))$ is the {\em virtual fundamental class}. If $\mathcal{X} = \IG^{ev}$ (or, more generally, $\mathcal{X} = G/P$) then the moduli space $\Mb_{0,r}(\mathcal{X},d)$ is an irreducible algebraic variety \cite{kim.pandharipande, thomsen}, and the virtual fundamental class coincides to the fundamental class. In the (non-homogeneous) case $\mathcal{X} = \IG$, and when $d=1$, Pech used obstruction theory to prove the following (cf.~\cite[Proposition 2.15]{pech:thesis}; for $k=2$ \cite[Proposition 13]{pech:quantum}):

\begin{prop}\label{prop:vfund} Let $r=1,2, 3$. Then the moduli space of stable maps $\overline{\mathcal{M}}_{0,r}(\IG(k,2n+1),1)$ is a smooth, irreducible, algebraic variety of complex dimension $k(2n+1-k)-\frac{k(k-1)}{2}+(2n+2-k)+r-3$. \end{prop}
 
The GW invariants satisfy the ``divisor axiom" property: if $[D]_T \in H^2_T(\mathcal{X})$ is a class in complex codimension $1$ then for any $\gamma_2, \ldots , \gamma_r \in H^*_T(\mathcal{X})$, \begin{equation}\label{E:divax} \langle [D]_T, \gamma_2, \ldots , \gamma_r \rangle_d = ([D] \cap d) \langle \gamma_2, \ldots , \gamma_r \rangle_d \/. \end{equation}
\subsection{The (equivariant) quantum cohomology ring} The quantum cohomology ring $\QH^(\IG)$ of $\IG:= \IG(k, 2n+1)$ is a graded $\Z(pt)[q]$-algebra with a $\Z[q]$-basis given by Schubert classes $[X(u)]$, where $u \in W^P \cap W^{odd}$.
The multiplication is given by \[ [X(u)] \star [X(v)] = \sum_{d \ge 0; w \in W^P} q^d c_{u,v}^{w,d} [X(w)] \/, \] where $c_{u,v}^{w,d} = \langle [X(u)], [X(v)], [X(w)]^\vee \rangle_d$ is the GW invariant. The degree of $q$ is 
\[\deg q= \int_{[X(Div)]^{\vee}} c_1(T_{\IG(k,2n+1)})=2n+2-k \/, \] by Lemma \ref{lemma:degq} above. The grading is equivalent to the requirement that \[ \textrm{codim}~ X(u) + \textrm{codim} ~X(v) = \textrm{codim}~ X(w) + d \cdot \deg q \/. \] The quantum cohomology ring is a deformation of the ordinary cohomology ring, in the sense that if one makes $q=0$ then one recovers the multiplication in $H^*_T(\IG)$.

As before there is an equivariant version of the quantum cohomology ring, denoted $\QH^*_T(\IG)$, which deforms the multiplication in $H^*_T(\IG)$. This is a graded, free algebra over $H^*_T(pt)[q]$ with a basis given by equivariant Schubert classes $[X(u)]_T$, where $u$ varies in $W^P \cap W^{odd}$. The multiplication is defined as before, using the equivariant GW invariants. The structure constants $c_{u,v}^{w,d} \in H^*_T(pt)$ are homogeneous polynomials of polynomial degree 
\[ \deg c_{u,v}^{w,d} = \textrm{codim}~ X(u) + \textrm{codim} ~X(v) - \textrm{codim}~ X(w) - d \cdot  \deg q \/. \]
If this degree equals $0$, then one recovers the structure constant from the ordinary (non-equivariant) quantum cohomology ring. 
\section{The moment graph}\label{s:moment} Sometimes called the GKM graph, the {\em moment graph} of a variety $X$ with a torus $T$ action has a vertex for each $T$-fixed locus of $X$, and an edge for each $1$-dimensional torus orbit. The description of the moment graphs for flag manifolds is well known, and it can be found e.g in \cite[Ch. XII]{kumar:kacmoody}. In this note we consider the moment graphs for $\IG:=\IG(k, 2n+1) \subset \IG^{ev}:= \IG(k, 2n+2)$. Let $P := P_k \subset \Sp_{2n+2}$ be the maximal parabolic for $\IG^{ev}$. The minimal length representatives in $w \in W^{P}$ are in one to one correspondence to sequences $1 \leq w(1) < \ldots < w(k) \leq \bar{1}$. Those corresponding to the odd-symplectic Grassmannian satisfy in addition that $w(i) \leq \bar{2}$ for $1 \leq i \leq n+1$. The moment graph of $\IG^{ev}$ has a vertex for each $w \in W^P$, and an edge $w \rightarrow ws_\alpha$ for each \[ \alpha \in R^+ \setminus R_{P_k}^+ = \{ t_i - t_j: 1 \leq i \leq k < j \leq n+1 \} \cup \{ t_i + t_j, 2 t_i : 1 \leq i <j \leq n+1, i \leq k  \} \/. \]
Geometrically, this edge corresponds to the unique torus-stable curve $C_\alpha(w)$ joining $w$ and $ws_\alpha$. This curve has degree $d$, where $\alpha^\vee + \Delta_P^\vee = d \alpha_k^\vee + \Delta_P^\vee$. The moment graph of $\IG$ is the full subgraph of that of $\IG^{ev}$ determined by the vertices $w \in W^P \cap W^{odd}$. Notice that the orbits of $T$ and $T_{2n+2}$ coincide, therefore we do not distinguish between the moment graphs for these tori. For later use, we list below the vertices adjacent to the identity element in the moment graph of $\IG^{ev}$, together with the degrees of the corresponding curves. Recall the convention $\bar s = 2n+3 - s$. For now we let $k>1$.
\begin{enumerate}
\item[(i)] $(1<2<\cdots<i-1<i+1<\cdots<k<j)$ where $k<j \leq n+1$; 
\item[(ii)] $(1<2<\cdots<i-1<i+1<\cdots<k< \bar j)$ where $n+1<\bar j \leq \overline{k+1}$;
\item[(iii)] $(1<2<\cdots<i-1<i+1<\cdots<j-1<j+1<\cdots<k<\bar j< \bar i)$;
\item[(iv)] $(1<2<\cdots<i-1<i+1<\cdots<k<\overline{i})$ .
\end{enumerate}
The edge in (i) corresponds to $\alpha = t_i - t_j$, those in (ii) and (iii) to $\alpha = t_i + t_j$ and that in (iv) to $\alpha = 2 t_i$. In particular, only the edge in (iii) has degree $2$, and the others have degree $1$. If $k=1$, the case (iii) does not apply, and the remaining vertices in cases (i), (ii) and (iv) are respectively $(j), (\bar j)$ and $(\bar 1)$. The figure below illustrates the moment graphs of $\IG(2,5)$ and $\IG(2,6)$.
\begin{figure}[h!]\label{figurea}
\caption{The figure is the moment graph of $\IG(2,6)$ without degree labels. The blue portion corresponds to vertices outside the Schubert variety $\IG(2,5)$. The red portion is inside the closed orbit $\IG(1,4)= \mathbb{P}^3$.}
\begin{tikzpicture}[-,>=stealth',shorten >=1pt,auto,node distance=1.9cm,
                    thick,main node/.style={draw,font=\sffamily\tiny\bfseries}]

  \node[main node] (1) {$(\bar{3}<\bar{2})$};
  \node[main node] (2) [below of=1] {$(3<\bar{2})$};
    \node[main node] (3) [below left of=2] {$(2<\bar{3})$};
  \node[main node] (4) [ red, below right of=2] {$(1<\bar{2})$};
  \node[main node] (5) [below of=3]{$(2<3)$};
  \node[main node] (6) [ red, below of=4]{$(1<\bar{3})$};
   \node[main node] (7) [red, below right of=5]{$(1<3)$};
   \node[main node] (8) [ red, below of=7]{$(1<2)$};
   \node[main node] (9) [blue,left of=2] {$(2<\bar{1})$};
    \node[main node] (10) [blue, left of=1] {$(3<\bar{1})$};
     \node[main node] (11) [blue, above of=1] {$(\overline{3}<\bar{1})$};
\node[main node] (12) [blue, above of=11] {$(\overline{2}<\bar{1})$};

  \path[every node/.style={font=\sffamily\large}]
    (8) edge [red] node {}  (7)
         edge node {} (5)
          edge [red] node {} (6)
          edge node {} (3)
          edge [red] node {} (4)
          edge [blue,dotted,bend left] node {} (9)
          edge [blue, dotted,bend right] node {} (12)
    (7) edge node {} (5)
          edge node {} (2)
          edge [red] node  {} (4) 
          edge [red] node {} (6)
            edge [blue,dotted] node {} (10)
          edge [blue, dotted,bend right] node {} (11)
    (6) edge node [right] {} (3)
          edge  node [right] {} (1)
          edge [red] node [right] {} (4)
             edge [blue,dotted] node {} (10)
          edge [blue, dotted] node {} (11)
    (5) edge  node {} (3)
          edge  node {} (1)
          edge node {} (2)
           edge [blue,dotted] node {} (9)
          edge [blue, dotted, bend left] node {} (10)
    (4) edge node {} (2)      
          edge  node {} (1)
           edge [blue,dotted] node {} (9)
          edge [blue, dotted,bend right] node {} (12)
    (3) edge  node {} (1)
          edge  node [right] {} (2)
          edge [blue,dotted] node {} (9)
          edge [blue, dotted] node {} (11)
    (2) edge node {} (1)
    edge [blue,dotted] node {} (10)
    edge [blue, dotted,bend right] node {} (12)
    (1)edge [blue,dotted] node {} (11)
          edge [blue, dotted,bend left] node {} (12)
       (9)edge [blue,dotted] node {} (10)
          edge [blue, dotted] node {} (11)
          edge [blue, dotted] node {} (12)
        (10)edge [blue,dotted] node {} (11)
          edge [blue, dotted] node {} (12)
          (11)edge [blue,dotted] node {} (12);
   
\end{tikzpicture}
\end{figure}
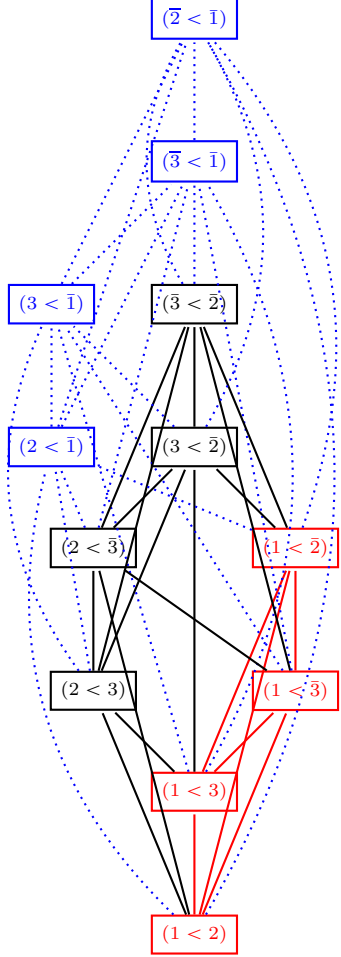 

\section{Curve Neighborhoods}\label{s:curvenbhds} Let $\mathcal{X} \in \{ \IG, \IG^{ev} \}$, let $d \in H_2(\mathcal{X})$ be an effective degree, and let $\Omega \subset \mathcal{X}$ be a closed subvariety. Consider the moduli space of stable maps $\Mb_{0,2}(\mathcal{X},d)$ with evaluation maps $\ev_1, \ev_2$. The {\em curve neighborhood} of $\Omega$ is the subscheme \[ \Gamma_d(\Omega):= \ev_2( \ev_1^{-1} \Omega) \subset \mathcal{X} \] endowed with the reduced scheme structure. This notion was introduced by Buch, Chaput, Mihalcea and Perrin \cite{BCMP:qkfin} to help study the quantum K theory ring of cominuscule Grassmannians. It was analyzed further for any homogeneous space by Buch and Mihalcea \cite{buch.m:nbhds}, in relation to $2$-point K-theoretic GW invariants, and to a new proof of the quantum Chevalley formula. Often, estimates for the dimension of the curve neighborhoods provide vanishing conditions for certain GW invariants. In this paper we will use this technique to prove vanishing of ``Chevalley" GW invariants of degree $d \ge 2$ in $\IG$.

We start with the observation (going back to \cite{BCMP:qkfin}) that if $\Omega$ is a Schubert variety, then $\Gamma_d(\Omega)$ must be a (finite) union of Schubert varieties, stable under the same Borel subgroup. This follows because $\Omega$ is stable under the appropriate Borel subgroup, and $\ev_1, \ev_2$ are proper, equivariant maps; thus $\Gamma_d(\Omega)$ is closed and Borel stable. Further, it was proved in \cite{BCMP:qkfin} that the curve neighborhood $\Gamma_d(Y(w))$ of any Schubert variety is again a Schubert variety. This Schubert variety was described in \cite{buch.m:nbhds}: $\Gamma_d(Y(w)) = Y(w \cdot z_d W_P )$, where $z_d \in W$ is {\em defined} by the condition that $\Gamma_d(1.P) = Y(z_d W_P)$. We recall next a recursive formula for $z_d$. Recall also that $q_{2n+2}$ denotes the quantum parameter for $\QH^*(\IG^{ev})$ and it has degree $2n+3 -k$. The maximal elements of the set $\{\beta \in R^+ \setminus R_{P}^+: \beta^\vee+\Delta_{P}^\vee \leq d \}$ are called {\it maximal roots} of $d$. The root $\beta \in R^+ \setminus R_{P}^+$ is called {\it ${P}$-cosmall} if $\beta$ is a maximal root of $\beta^{\vee}+\Delta_{P}^{\vee} \in H_2(\IG(k,2n+2))$. In type $C_{n+1}$, the $P$-cosmall roots are the roots $2t_i$ for $1 \leq i \leq n+1$, and $t_i-t_j$ for $1 \leq i<j \leq n+1$. The following follows from \cite[Corollary 4.12, Theorem 6.2, Theorem 5.1, and Theorem 7.2]{buch.m:nbhds}.

\begin{prop}\label{ggg} Let $d \in H_2(\IG(k,2n+2))$ be an effective degree and $w \in W^P$. Then the following hold:
\begin{enumerate}
\item If $\alpha \in R^+ -R_{P}^+$ is a maximal root of $d$, then $s_\alpha \cdot z_{d - \alpha^\vee} W_P = z_d W_P$; 
\item $\dim \Gamma_d(Y(w)) \leq \ell(w) + \ell(z_d W_P) \leq \ell(w) + d \cdot \deg q_{2n+2} -1 $. Furthermore, if the second equality occurs then $d=\alpha^\vee+\mathbb{Z} \Delta_{P}^{\vee}$ and $\alpha$ is a $P$-cosmall root.
\end{enumerate}
\end{prop}
\begin{cor}\label{cor:evest} (a) If $k >1$ then there is an equality $z_1 W_P = s_{2 t_1} W_P$ and the minimal length representative of $z_1 W_P$ is $(2<3<\cdots<k<\overline{1})$. 

(b) There is an inequality $\ell(z_d W_P) \leq d \deg q_{2n+2} -1 $ with equality if and only if $d=1$. 

(c) If $k>1$ and $d=2$ then $z_2 W_P = s_{t_1+t_2} W_P$ and $\ell(z_2 W_P) = 2 \deg q_{2n+2} -3$. 

(d) If $k=1$ then $z_1W_P=z_2W_P$ and $\ell(z_2W_P)=2n+1<2 \deg q_{2n+2} -3$.
\end{cor}

\begin{proof} The first part follows directly from the part (1) of the proposition. The equality in (b) follows by direct calculation of $\ell(s_{2t_1} W_P)$, using its minimal length representative. A calculation of degrees of roots shows that no degree $d  \geq 2$ can be the degree of a cosmall root, thus equality cannot occur in this case.

For part (c), notice that $2t_1$ is a maximal root of $d=1$, therefore $z_1 W_P = s_{2 t_1} W_P$. By the recursion in Proposition \ref{ggg} we obtain $z_2W_P=s_{2t_1}\cdot s_{2t_1}W_P$. Now observe the following:  \begin{eqnarray*}
s_{2t_1}\cdot s_{2t_1}&=&(s_1 \dots s_{n+1} \dots s_1)\cdot (s_1 \dots s_{n+1} \dots s_1)
= s_1 \cdot s_{2t_{2}} \cdot s_1 \cdot s_{2t_{2}} \cdot s_1 =s_1 \cdot s_{t_1+t_{2}} \cdot s_1\\
&=&s_{t_1+t_{2}} \cdot s_1 \cdot s_1 =s_{t_1+t_{2}} \cdot s_1 \/.
\end{eqnarray*}
Since $s_1 \in W_P$, the above shows that $z_2 W_P = s_{t_1+ t_2} W_P$ as claimed. The equality $\ell(z_2 W_P) = 2 \deg q_{2n+2} -3$ follows by a direct calculation, using the minimal length representative of $s_{t_1+t_2}W_P$. Finally, part $(d)$ follows from the observation that if $k=1$, then $\IG(1, 2n+2) = \mathbb{P}^{2n+1}$, and then $\Gamma_1(id) = \mathbb{P}^{2n+1}$, thus $\ell(z_1W_P)= \dim \IG(1,2n+2)$.
\end{proof}

In what follows we give estimates for the dimension of the curve neighborhoods of Schubert varieties $X(w) \subset \IG$, using known estimates for the dimension in the even case. We will need the following lemma.

\begin{lemma}
\label{Lem11}
Let $w=(w(1)<w(2)<\cdots<w(k)) \in W^P$. Then $\ell( w \cdot z_1 W_P) = \ell(w) + \ell(z_1W_P)$ if and only if $w(1)=1$. In particular, $\dim \Gamma_1(Y(w)) = \ell(w) + \deg q_{2n+2} -1$ if and only if $Y(w) \subset X_c$ is a Schubert variety in the closed orbit of $\IG$.  
\end{lemma}
\begin{proof} Let $\overline{z}_1 := (2<3<\cdots<k<\overline{1}) \in W^P$ be the minimal length representative of $z_1 W_P$. By Lemma \ref{lemma:bound}, $\ell( w \cdot z_1 W_P) = \ell(w) + \ell(z_1W_P)$ if and only if the product $w \overline{z}_1$ is reduced and it is a minimal length representative. We calculate $w \overline{z}_1 = (w(2), w(3) , \ldots , w(k), \overline{w(1)}, w(k+1), \ldots , w(n+1))$. If $w(1) =1$ then clearly $w \overline{z}_1 \in W^P$, and one checks $\ell(w \overline{z}_1) = \ell(w) + \ell(\overline{z}_1)$. Conversely, if $w \cdot \overline{z_1} =w \overline{z}_1 \in W^P$, then one uses the bijection between $W^P$ and the strict partitions described in \S \ref{ss:dictionary} below to calculate that \[ \ell (w \overline{z}_1)-\ell(w) = 2n+4-2w(1) -k+\#\{j: w(1)+w(j) > 2n+3 \} \/. \] The length condition forces $w(1) =1$.~(For a similar proof see Proposition \ref{prop:dictionary} below).\end{proof}
%
%

\subsection{Curve neighborhoods for $\IG(k,2n+1)$} Let $w \in W^P \cap W^{odd}$ and let $d \in H_2(\IG)$ be an effective degree. As mentioned above, the curve neighborhood $\Gamma_d(X(w))$ of $X(w)$ is a closed, $B$-stable subvariety of $\IG$, therefore it must be an union of Schubert varieties: \[ \Gamma_d(X(w))=X(w^1) \cup  \cdots \cup X(w^r)\] where $w^i \in W^P \cap W^{odd}$. As noticed in \cite[\S 5.2]{buch.m:nbhds} and \cite[Cor. 5.5]{mare.mihalcea}, the permutations $w^i$ can be determined combinatorially from the moment graph. 

\begin{prop}
\label{thm11} Let $w \in W^P \cap W^{odd}$. In the moment graph of $\IG(k,2n+1)$, let $\{v^1, \cdots, v^s\}$ be the maximal vertices in the moment graph which can be reached from any $u \leq w$ using a path of degree $d$ or less. Then $\Gamma_d(X(w))=X(v^1) \cup \cdots \cup X(v^s)$.
\end{prop}
\begin{proof}
Let $Z_{w,d}=X(v^1) \cup \cdots \cup X(v^s)$. Let $v:= v^i \in Z_{w,d}$ be one of the maximal $T$-fixed points. By the definition of $v^i$'s and the moment graph there exists a chain of $T$-stable rational curves of degree less than or equal to $d$ joining $u \leq w$ to $v$. It follows that $v \in \Gamma_d(X(w))$, thus $X(v) \subset \Gamma_d(X(w))$, whence $Z_{w,d} \subset \Gamma_d(X(w))$.

For the converse inclusion, let $v \in \Gamma_d(X(w))$ be a $T$-fixed point. By \cite[Lemma 5.3]{mare.mihalcea} there exists a $T$-stable curve joining a fixed point $u \in X(w)$ to $v$. This curve corresponds to a path in the moment graph of $\IG(k,2n+1)$, thus $v \in Z_{w,d}$. Since Bruhat order is compatible with inclusion of Schubert varieties, this completes the proof.
\end{proof}

In what follows we will obtain estimates for the dimension of the curve neighborhood $\Gamma_d(X(w))$, using estimates obtained in the even case. We start with the observation that the ``odd" curve neighborhoods are proper subvarieties of the ``even" ones.

\begin{lemma} \label{Lemma:comp} Let $w \in W^P \cap W^{odd}$ and $d \ge 1$ an effective degree. Then there is a strict inclusion 
 $\Gamma_d(X(w)) \subsetneq \Gamma_d(Y(w))$.
\end{lemma}

\begin{proof} Consider the identity $1.P \in \Gamma_d(X(w))$. There is a $T$-stable degree $1$ curve (i.e. a line) in $\IG(k, 2n+2)$ that contains the $T$-fixed points $1.P $ and $(2<3<\cdots<\bar{1}) \in \IG(k,2n+2) \setminus \IG(k,2n+1)$. \end{proof}

The next result is the key technical requirement needed for the vanishing of certain Chevalley GW invariants.

\begin{thm} \label{Lem10}
Let $w \in W^P \cap W^{odd}$. Then the following inequalities hold:
\begin{eqnarray*}
\dim \Gamma_1(X(w))-\dim X(w) &\leq& \deg q-1\\
\dim \Gamma_d(X(w))-\dim X(w) &<& d \deg q-1 \mbox{ for all } d \geq 2
\end{eqnarray*}
Further, if the Schubert variety $X(w)$ is not contained in the closed orbit $X_c$ of $\IG$ then \[ \dim \Gamma_1(X(w))-\dim X(w) <\deg q-1 \/. \]
\end{thm}

\begin{proof} Recall that $\deg q_{2n+2}=\deg q +1$. If $d=1$, by Lemma \ref{Lemma:comp} and Proposition \ref{ggg} \begin{eqnarray*}
&&\dim \Gamma_1(X(w))+1-\dim X(w) \leq \dim \Gamma_1(Y(w))-\dim Y(w) \leq \deg q_{2n+2}-1 \end{eqnarray*} thus $\dim \Gamma_1(X(w))-\dim X(w) \leq \deg q_{2n+2}-2=\deg q-1$. Let now $d=2$. If $k >1$ then by Lemma \ref{Lemma:comp} and Corollary \ref{cor:evest} we obtain
\[ \dim \Gamma_2(X(w)) - \dim X(w) \leq \dim \Gamma_2(Y(w)) -1 - \ell(w) \leq \ell(z_2 W_P) -1 \leq 2 \deg q_{2n+2} -4 < 2 \deg q -1 \/. \]
For arbitrary $d$, let $\Gamma_d(X(w))=X(v^1) \cup  \cdots \cup X(v^s)$. Then each $v^i$ is joined to some $u^i \leq w$ in the moment graph of $\IG(k,2n+1)$ by $j$ edges of degrees $d_i \in \{ 1 ,2 \}$, where $\sum_{i=1}^j{d_i} \leq d$. By applying repeatedly the estimates for $d=1,2$ we have \[\dim X(v^i) - \dim X(w) \leq \dim X(v^i)-\dim X(u^i) \leq \sum_{i=1}^j (d_i \deg q-1) \leq d \deg q-j. \] If $j \geq 2$ then the result holds, and if $j=1$ then necessarily $d \in \{ 1, 2 \}$, a case treated before. 
This proves the first two inequalities. For the last inequality, we notice that the hypothesis implies that $w$ is determined by a sequence $(w(1)<\cdots<w(k))$ such that $w(1)>1$. Then by Lemma \ref{Lem11} combined with Proposition \ref{ggg} we obtain \[ \begin{split} \dim \Gamma_1(X(w)) - \dim X(w) \leq \dim \Gamma_1(Y(w)) - 1 - \ell(w) <  \deg q_{2n+2} -2 = \deg q -1 \/. \end{split} \] This finishes the proof. \end{proof}

\section{Vanishing of Chevalley Gromov-Witten invariants}\label{s:vanishing} The main result of this section is the following.

\begin{thm}\label{thm:vanishing} Let $d\geq 1$ be a degree in $H_2(\IG(k, 2n+1))$. Let $X(v), X (w) \subset \IG(k, 2n+1)$ be two Schubert varieties and $X(Div)$ the Schubert divisor. If $\dim \Gamma_d(X(v)) < \ell(v) + d \deg q -1 $ then the equivariant GW invariant 
\[ \langle [X(Div)]_T, [X(v)]_T, [X(w)]_T^\vee \rangle_d = 0 \/. \] In particular, the equivariant Gromov-Witten invariant above vanishes if either $d\geq 2$ or if $d=1$ and $X(v)$ is not included in the closed orbit $X_c \subset \IG(k, 2n+1)$. \end{thm}

\begin{proof} By the divisor axiom \[ \langle [X(Div)]_T, [X(v)]_T, [X(w)]_T^\vee \rangle_d = d \langle [X(v)]_T, [X(w)]_T^\vee \rangle_d \/. \] By definition, \[\begin{split} \langle [X(v)]_T, [X(w)]_T^\vee \rangle_d = \int_{[\Mb_{0,2}(\IG(k, 2n+1),d)^{vir}]_T}^T \ev_1^* [X(v)]_T \cup \ev_2^*[X(w)]_T^\vee \\ = \int_{\IG(k, 2n+1)}^T [X(w)]_T^\vee \cap (\ev_2)_* ( \ev_1^*[X(v)]_T \cap [\Mb_{0,2}(\IG(k, 2n+1),d)^{vir}]_T) \/. \end{split} \]
The cycle $(\ev_2) ( \ev_1^{-1} [X(v)])$ is supported on the curve neighborhood $\Gamma_d(X(v))$, and the push-forward $(\ev_2)_* ( \ev_1^*[X(v)]_T \cap [\Mb_{0,2}(\IG(k, 2n+1),d)^{vir}]_T)$ is non-zero only if the curve neighborhood has components of dimension \[ \textrm{expdim}~ \Mb_{0,2}(\IG(k, 2n+1)) - \textrm{codim}~ X(v) = \deg q^d -1 + \ell(v) \/. \] However, the hypothesis implies that $\dim \Gamma_d(X(v))$ is strictly less than this quantity. The last statement follows from Theorem \ref{Lem10}. \end{proof}

\section{Lines in $\IG(k,2n+1)$}\label{s:lines} As before, we set $\IG:= \IG(k, 2n+1)$. If $k \neq n+1$ then $\Sp_{2n+1}$ acts with two orbits: $X^o$ (the open orbit) and $X_c$ (the closed orbit). If $k= n+1$ the space $\IG$ is homogeneous under $\Sp_{2n+1}$, and $\IG = X_c$ is isomorphic to the Lagrangian Grassmannian $\IG(n, 2n)$. All statements remain true in this case after making $X^o = \emptyset$, with almost identical proofs.

 According to Theorem \ref{thm:vanishing}, the only equivariant GW invariants $\langle [X(Div)]_T, [X(v)]_T, [X(w)]_T^\vee \rangle_d$ which maybe non-zero are those when $d=1$ and the Schubert variety $X(v)$ is included in the closed orbit $X_c \simeq \IG(k-1, 2n)$. To calculate these invariants, we will analyze the geometry of the moduli spaces of stable maps $\Mb_{0,r} (\IG,1) \to \IG$ where $r=1,2$, and the geometry of the {\em Gromov-Witten varieties} \[ GW_1(w):= \ev_1^{-1} (X(w)) \subset \Mb_{0,2}(\IG, 1) \/. \] For $X(w) \subset X_c$, we will show that $GW_1(w)$ is a scheme which has $2$ irreducible, generically reduced, components. One component parametrizes lines in $\IG$ contained in the closed orbit $X_c$, and the other those lines which intersect the open orbit $X^\circ$. The restriction of the evaluation maps to each of these components will be a surjective map, which is either birational, or it has general fiber of positive dimension. We will deduce from this that the curve neighborhood $\Gamma_1(X(w))$ has two components, and that if non-zero, the GW invariant is equal to $1$ precisely in the cases when $[X(w)]^\vee$ is Poincar{\'e} dual to one of these components. 

From now on, a {\em line} in $X$ will mean an irreducible, reduced, curve of degree $1$. Recall that there is a sequence of embeddings \[ \IG(k, 2n+1) \subset \IG(k, 2n+2) \subset \Gr(k, 2n+2) \subset \mathbb{P}( \bigwedge^{k} \mathbb{C}^{2n+2}) \] where the last is the Pl{\"u}cker embedding. The image of a line in $\IG$ under the composition of these embeddings is a projective line. Indeed, a calculation in coordinates shows that the image of the Schubert curve in $\IG$ is the Schubert curve in $\Gr(k, 2n+2)$, and the image of this Schubert curve is a projective line. 

Let $r=1,2,3$. We recall from Prop. \ref{prop:vfund} that $\Mb_{0,r}(\IG, 1)$ is a non-singular, irreducible scheme of dimension \[ \dim \Mb_{0,r}(\IG, 1)= \dim \IG + \deg q + r-3 = k(2n+1-k) - \frac{k(k-1)}{2} + 2n+2 - k + r-3 \/. \] 
 There is a natural isomorphism $\Mb_{0,1}(\Gr(k, 2n+1),1) \simeq \mathrm{Fl}(k-1,k,k+1; 2n+1)$ (a $3$-step flag variety) such that the evaluation map $\ev_1$ is the projection $\pi_k:\Fl(k-1,k,k+1;2n+1) \to \Gr(k, 2n+1)$. To see the isomorphism explicitly, one can use e.g.~the kernel-span technique of Buch \cite{buch} to observe that to any line $L \subset \Gr$ one can associate its kernel $K:= \bigcap_{V \in L} V$ and its span $S:= \Span\{ V: V \in L \}$, which have dimension $k-1$, respectively $k+1$. Then the pointed line $(p \in L)$ is sent to $(\ker L, p, \Span L)$. Although logically not needed in what follows, we remark that one can identify $\Mb_{0,1}(\IG,1)$ to a subvariety of the three-step flag variety, by noticing that if $V_{k-1} \in \IG(k-1, 2n+1)$ is a kernel of a line, then a triple $(V_{k-1} \subset V_k \subset V_{k+1})$ corresponds to a line in $\IG$ if and only if $V_{k-1}$ is isotropic and $ V_{k+1} \subset V_{k-1}^\perp$. Therefore $\Mb_{0,1}(\IG,1)$ can be identified set theoretically with \[ \Mb_{0,1}(\IG,1) \simeq \{ (V_{k-1} \subset V_k \subset V_{k+1}): V_{k-1} \in \IG(k-1, 2n+1), V_{k+1} \subset V_{k-1}^\perp \} .\] 
Under this identification, $\ev_1$ corresponds to the projection to the component $V_k$. 
\subsection{Lines intersecting the open orbit $X^\circ$} Consider the open subvariety $\mathcal{M}^\circ \subset \Mb_{0,1}(\IG,1)$ parametrizing $1$-pointed lines intersecting the open orbit $X^\circ \subset \IG$: \[ \mathcal{M}^\circ := \{ (p, L): L \cap X^\circ \neq \emptyset \} \/. \] Since the kernel of a line $L$ intersecting $X^\circ$ cannot contain $\e_1$ (which spans the kernel of the odd symplectic form), the variety $\mathcal{M}^\circ$ can be realized as the flag bundle $\mathcal{F}\ell(1,2; \mathcal{S}_{k-1}^\perp/ \mathcal{S}_{k-1})$ over the open orbit $\IG(k-1,2n+1)^\circ$, where $\mathcal{S}_{k-1}$ denotes the tautological subbundle. In this case $rank(\mathcal{S}_{k-1}^\perp) = 2n+1 - (k-1)$. 
Let $\pi: \mathcal{M}^\circ \to X$ denote the natural projection map. Key to the calculation of the GW invariants is the following result, analyzing the geometry of the fibres of $\pi$. 

\begin{thm}\label{thm:pigeom} 
(a) The natural projection map $\pi: \mathcal{M}^\circ \to \IG(k, 2n+1)$ is surjective, and all its fibers are irreducible, generically smooth, of dimension $\dim \mathcal{M}^\circ - \dim \IG(k, 2n+1)$.

(b) The inverse image $\pi^{-1}(X_c)$ is isomorphic to an $\Sp_{2n+1}$ orbit in $\IF(k-1, k, k+1; 2n+1)$. In particular, it is smooth and irreducible.\end{thm}
Before proving the theorem, we recall the description of the $Sp_{2n+1}$-orbits of the odd-symplectic 3-step partial flag variety $\IF(k-1,k,k+1;2n+1)$:
\begin{eqnarray*}
K_1&=&\{V_{k-1} \subset V_k \subset V_{k+1} \in \IF(k-1,k,k+1;2n+1):{\bf e}_1 \in V_{k-1} \}\\
K_2&=&\{V_{k-1} \subset V_k \subset V_{k+1} \in \IF(k-1,k,k+1;2n+1):{\bf e}_1 \in V_{k}, {\bf e}_1 \notin V_{k-1} \}\\
K_3&=&\{V_{k-1} \subset V_k \subset V_{k+1} \in \IF(k-1,k,k+1;2n+1):{\bf e}_1 \in V_{k+1}, {\bf e}_1 \notin V_{k} \}\\
K_4&=&\{V_{k-1} \subset V_k \subset V_{k+1} \in \IF(k-1,k,k+1;2n+1): {\bf e}_1 \notin V_{k+1} \}.
\end{eqnarray*}
We also need the following lemma:
\begin{lemma} \label{iso} Let $L$ be a line such that $L \cap X_c \neq \emptyset$ and $L \cap X^{\circ} \neq \emptyset$. Then $\mathrm{Span}~L$ is an isotropic subspace in $\C^{2n+1}$. 
\end{lemma}
\begin{proof} Let $x \in L \cap X_c$ and $y \in L \cap X_1^{\circ}$. Since $x \in X_c$ and $y \notin X_c$ we can choose a basis $\{{\bf e}_1,x_1,\cdots,x_{k-1}\}$ for $x$ such that $\{ x_1,\cdots, x_{k-1} \}$ is a basis for $x \cap y$ and choose a basis $\{x_1,\cdots, x_{k-1},f \}$ for $y$. Then $\{{\bf{e}}_1,x_1, \cdots, x_{k-1}, f\}$ is a basis for $\mathrm{Span}~L = \langle x ,y \rangle$. Clearly $\langle x_i, f \rangle = 0$ and since ${\bf e}_1 \in \ker \omega$ it follows that $\langle {\bf e}_1, f \rangle = 0$. This finishes the proof.\end{proof}

We note that this is the best result possible. For instance let $n=k=2$ and consider the line that contains the $T$-fixed points $(1<3)$ and $(1<\overline{3})$ (this is a line included in the closed orbit $X_c \simeq \IG(1, 4) \simeq \mathbb{P}^3$). Then $\Span L= \left<\e_1,\e_3, \e_{\overline{3}} \right>$ is not isotropic, because $\omega(\e_3,\e_{\overline{3}})=1$. Similarly, the line joining $(2<3)$ to $(3< \bar 2)$ (a line in the open orbit) has again non-isotropic span; see figure 2 below for more examples. 

We will need to calculate $\dim K_2$. For that, observe that to construct a triple in $K_2$ one first chooses $V_k \in \IG(k, 2n+1)_c \simeq \IG(k-1, 2n)$, then $V_{k-1}$ in an open set in $\Gr(k-1, V_k)$, and then finally an open set of $V_{k+1} \in \Gr(1, V_{k-1}^\perp/V_k) $. (The spaces $V_{k+1}$ obtained this way are automatically isotropic, because $\e_1 \in V_k$.) This yields
\begin{equation}\label{dimK2} \dim K_2 = \dim \IG(k-1,2n)+ (k-1) + (2n-2k+1) = (k-1)(2n-k+1)- \frac{(k-1)(k-2)}{2}+2n-k \/. \end{equation}  
\begin{proof}[Proof of Theorem \ref{thm:pigeom}] The definition of $\mathcal{M}^\circ$ implies that $\pi$ is surjective over the open orbit $X^\circ$. By \cite[Prop. 2.3]{BCMP:qkfin} this is a locally trivial fibration, and because both $\mathcal{M}^\circ$ and $X^\circ$ are smooth and irreducible, it follows that the fibers over $X^\circ$ are also smooth and irreducible. Notice that the same result implies that $\pi^{-1}(X_c)$ is a locally trivial fibration over $X_c$. To prove (a) it remains to show that the fibre $\pi^{-1}(1.P)$ is nonempty, irreducible and generically smooth.

As explained in  \S \ref{s:moment}, there is a line joining $1.P$ to $\langle e_2, e_3, \ldots , e_k, e_{k+1} \rangle \in X^\circ$. Thus $\pi^{-1}(1.P) \neq \emptyset$. We prove next that the reduced support $(\pi^{-1}(X_c))_{red}$ is irreducible, which implies that 
$\pi^{-1}(X_c)$ is again irreducible. Then we will use a local calculation to find an open dense set of $\pi^{-1}(1.P)$ where it is smooth. In the process we will simultaneously prove both (a) and (b). 

To start, there is a bijective morphism $K_2 \to (\pi^{-1}(X^c))_{red}$ defined as follows: to each pointed line $(p \in L)$ in $\IG$ such that $ p \in L \cap X_c$ and $L \cap X^{\circ} \neq \emptyset$ one associates the element $(\ker L, p, \Span L) \in \IF^{odd}(k-1, k , k+1; 2n+1)$. (The fact that the span of $L$ is isotropic follows from Lemma \ref{iso}.) Conversely, to each element $(V_{k-1} \subset V_k \subset V_{k+1}) \in K_2$ one associates the line $L:= \mathbb{P}(V_{k+1} / V_{k-1})$ and the point $V_k \in X_c$. Since $V_{k+1}$ is isotropic it follows that $L$ is a line in $\IG$; the condition $\e_1 \notin V_{k-1}$ implies that $L$ cannot be included in the closed orbit, so $L \cap X^\circ \neq \emptyset$. The fact that this is an algebraic morphism follows e.g.~because $K_2$ is an orbit of $\Sp_{2n+1}$.~This proves that $\pi^{-1}(X^c)$ is irreducible. Since $\pi^{-1}(X_c) \to X_c$ is a locally trivial fibration, it follows that $\pi^{-1}(1.P)$ is irreducible, and that it has dimension \[ \dim \pi^{-1}(1.P) = \dim K_2 - \dim X_c = 2n-k = \dim \mathcal{M}^\circ - \dim \IG \/. \] 
Turning to smoothness, we will show that there exist open sets $U_1 \subset \IG$ and $U_2 \subset \mathcal{M}^\circ$ such that $1.P \in U_1$, $U_2 \subset \pi^{-1}(U_1)$, $U_i$'s are isomorphic to open sets in some affine spaces $\mathbb{A}^{N_i} $, $i=1,2$ (for appropriate $N_i$), and such that the induced map $U_2 \to U_1$ is smooth. Using the coordinate charts in $\Gr(k, 2n+1)$ one defines the open set $U_1$ around $1.P$ to be given
by the column space of the matrix $\begin{pmatrix} I_k \\ A \end{pmatrix}$ where $A= (a_{i,j})$ is a $(2n+1 - k) \times k$ matrix. The isotropy constraints on the coordinates can be arranged in a triangular system with equations of the form $a_{i,j} + {\rm quadratic~terms} = 0$, where $1 \le j \le k-1$ and $ \overline{k} \le i \le \overline{j+1}$. This implies that $U_1$ is isomorphic to an affine space $\mathbb{A}^{\dim \IG}$. 
%

To define $U_2$, observe that an open set in the dual projective space of codimension $1$ subspaces $V_{k-1} \subset V_k= \langle v_1, \ldots , v_k \rangle \in U_1$ where $\e_1 \notin V_{k-1}$ is given by $\left< v_i+c_iv_1:  2 \leq i \leq k, c_i \in \mathbb{C} \right>$. Then an open set $U_2$ around triples containing such $V_k$ is given by the column span of the matrix {\color{black} $C:= (C_1 | C_2 | \ldots | C_k | C_{k+1})$ where $C_i$ are column vectors in $\C^{2n+1}$, defined as follows: \[ C_k = \e_1 + \sum_{j=k+1}^{\bar{2}} a_{j,1} \e_j \/; \quad  C_{k+1} = \e_{k+1} + \sum_{j = k+2}^{\bar{2}} d_j \e_j \/, \] and
\[ C_i = c_{i+1} \e_1 + \e_{i+1} + \sum_{j=k+1}^{\bar{2}} (a_{j,i+1} + c_{i+1} a_{j,1}) \e_j \/; \quad 1 \le i \le k-1 \/. \] By definition, the span $\Sigma_{k-1}:= \mathrm{span} (C_1, \ldots , C_{k-1}) $ of the first $k-1$ columns is an isotropic subspace, and the column vectors $C_k$ and $C_{k+1}$ are perpendicular to $\Sigma_{k-1}$; the projection to $\IG(k-1, 2n+1)^\circ$ sends the matrix $C$ to $\Sigma_{k-1}$. The
isotropy conditions translate into linear constraints which determine the coordinates $d_{\bar{2}}, \ldots , d_{\bar{k}}$ and the coordinates $a_{i,j}$, where $1 \le j \le k-1$ and $ \overline{k} \le i \le \overline{j+1}$ (these latter constraints are the same as those from $U_1$). There are $(2n+2-k)(k-1)+4n-2k+1$ coordinates and $\frac{k(k-1)}{2}+k-1$ of them are determined from linear constraints; this shows that $U_2 \simeq \mathbb{A}^{\dim \mathcal{M}^\circ}$. In these coordinates the map $\pi_{U_2}: U_2 \to U_1$ becomes the linear map given by $c_i \mapsto 0$ and $d_i \mapsto 0$. In particular, this map is smooth, and the fiber $\pi^{-1}(1.P) \cap U_2$ is smooth. This finishes the proof.} \end{proof}

\begin{example}\label{ex:coords} We illustrate the local calculation for $k=n=3$. The open sets $U_1 \subset \IG(3,7)$ and $U_2 \subset \mathcal{M}^\circ \simeq \mathcal{F}\ell(1,2; \mathcal{S}_2^\perp / \mathcal{S}_2)$ (a flag bundle over $\IG(2,7)^\circ$) are given by:
\[ U_2= \left( \begin{array}{c c | c| c}  c_2 & c_3 & 1 & 0 \\ 1 &  0 & 0 & 0 \\ 0 & 1 & 0 & 0 \\ a_{4,2}+c_2 a_{4,1} & a_{4,3} + c_3 a_{4,1} & a_{4,1} & 1 \\ a_{\bar{4},2} + c_2 a_{\bar{4},1} & a_{\bar{4},3} + c_3 a_{\bar{4},1} & a_{\bar{4},1} & d_{\bar{4}} \\  a_{\bar{3},2}^\bullet + c_2 a_{\bar{3},1}^\bullet & a_{\bar{3},3} + c_3 a_{\bar{3},1}^\bullet & a_{\bar{3},1}^\bullet & d_{\bar{3}}^\bullet \\ a_{\bar{2},2} + c_2 a_{\bar{2},1}^\bullet & a_{\bar{2},3} + c_3 a_{\bar{2},1}^\bullet & a_{\bar{2},1}^\bullet & d_{\bar{2}}^\bullet \end{array} \right) \overset{\pi_{U_2}}\longrightarrow U_1 = \begin{pmatrix} 1 & 0 & 0 \\ 0 & 1 & 0 \\ 0 & 0 & 1 \\ a_{4,1} & a_{4,2} & a_{4,3} \\ a_{\bar{4},1} & a_{\bar{4},2} & a_{\bar{4} ,3} \\ a_{\bar{3},1}^\bullet & a_{\bar{3},2}^\bullet & a_{\bar{3} ,3} \\ a_{\bar{2},1}^\bullet & a_{\bar{2},2} & a_{\bar{2} ,3} \end{pmatrix} \]
The coordinates with $\bullet$ are determined from linear equations, using the isotropy contraints. For instance, $a_{\bar{2},1}^\bullet$ in $U_1$ is determined by imposing that the first and second column are pependicular, i.e. \[ a_{\bar{2},1}^\bullet \cdot 1 + a_{\bar{4},1} \cdot a_{4,2} - a_{4,1} \cdot a_{\bar{4},2} = 0 \/. \] The third and fourth column vectors from $U_2$ are each perpendicular to the first two column vectors. The dimension of $U_2$ is $17$ (coordinates) - $5$ (linear constraints) $ = 12$, which equals $\dim \Mb_{0,1}(\IG(3,7),1)$, as claimed.
\end{example}


\subsection{Lines in the closed orbit} Set $\Mb:=\Mb_{0,1}(\IG,1)$ and consider the closed subvariety \[ \mathcal{M}_c:=\Mb \setminus \mathcal{M}^\circ = \{ (p,L) \in \Mb_{0,1}(\IG, 1): L \subset X_c \} \/, \] which consists of lines included in the closed orbit. In terms of triples of flags this consists of triples $(V_{k-1} \subset V_k \subset V_{k+1})$ such that $V_{k-1}$ belongs to the closed orbit in $\IG(k-1, 2n+1)$ (i.e $\e_1 \in V_{k-1}$), and $V_{k+1} \subset V_{k-1}^\perp$. Since $\e_1$ spans the kernel of the odd-symplectic form $\omega$, this is a smooth subvariety of $\Fl(k-1, k, k+1; 2n+1)$, and the universal property for the moduli space of stable maps gives a bijective morphism $\mathcal{M}_c \to \Mb_{0,1}(X_c, 1)$. It follows that $\mathcal{M}_c$ is isomorphic to the moduli space $\Mb_{0,1}(X_c, 1)$. Recall that $X_c$ is isomorphic to the homogeneous space $\IG(k-1, 2n)$, thus the moduli space $\Mb_{0,1}(X_c, 1)$ is a smooth, irreducible variety of dimension \[ \dim \mathcal{M}_c = \dim \IG(k-1, 2n) + 2n+1-(k-1) -2 =   \dim \IG(k-1, 2n) + 2n-k \/.\] (Note the coincidence $\dim \mathcal{M}_c = \dim K_2$.) We recall the following result, proved in Thm. 2.5 and Cor. 3.3 from \cite{BCMP:qkfin}:

\begin{lemma}\label{lemma:fibrec} For every $V \in X_c$, the fibre $\ev_1^{-1}(V)$ of the restricted map $\ev_1: \Mb_{0,1}(X_c,1) \to X_c$ is an irreducible, normal variety of dimension $\dim \mathcal{M}_c - \dim X_c$. \end{lemma}

We combine the previous lemma to Theorem \ref{thm:pigeom} to obtain the main result of this section.

\begin{thm}\label{thm:fibres} Consider the evaluation map $\ev_1: \Mb_{0,1}(\IG, 1) \to \IG$. Then the following hold:

(a) For any $V \in \IG$, the fibre $\ev_1^{-1}(V)$ is pure dimensional of dimension $\dim  \Mb_{0,1}(\IG, 1) - \dim \IG$, and each of its components is generically smooth. In particular, $\ev_1$ is flat.

(b) For any Schubert variety $X(w) \subset X_c$, the preimage $\ev_1^{-1}(X(w))$ has two irreducible components:
\[ \ev_1^{-1}(X(w)):= A_1 \cup A_2 \/, \] where $A_1$ is the closure of the subvariety of pointed lines $(p, L)$ such that $L \cap X^\circ \neq \emptyset$, and $A_2$ is the closed subscheme corresponding to $(p,L)$ such that $L$ is included in the closed orbit $X_c$. Further, each irreducible component is generically smooth of expected dimension $\dim \Mb_{0,1}(\IG,1) - \mathrm{codim}_{\IG} X(w)$. \end{thm} 

\begin{proof} Since $\Sp_{2n+1}$ acts transitively on the open orbit $X^\circ$, the morphism $\ev_1$ is flat, and the fibres have the stated dimension. Transitivity implies that all fibers over the closed orbit are isomorphic, thus it suffices to take $V= 1.P$. Let $F:= \ev_1^{-1}(1.P)$ be the fibre. Recall the notation $\mathcal{M}^\circ$ and $\mathcal{M}_c$.  Clearly $F$ can be written as the disjoint union $F= F^\circ \cup F_c$ where $F^\circ := F \cap  \mathcal{M}^\circ$ is open in $F$ and $F_c:=  F \setminus F^\circ$ is closed in $\mathcal{M}_c$. It follows from Theorem \ref{thm:pigeom} that $F^\circ$ is irreducible, generically reduced, and of the stated dimension. On the other side, Lemma \ref{lemma:fibrec} implies that $F_c$ is irreducible, reduced, of dimension \[ \dim F_c = \dim \mathcal{M}_c - \dim X_c = \dim \mathcal{M} - \dim \IG \/; \] (the last equality is a simple calculation). Therefore $F_c$ cannot in the closure of $F^\circ$, and the statements about $F$ hold.\begin{footnote} {Another way to see that $F_c \subsetneq \overline{F^\circ}$ is to notice that every line in $F^\circ$ has isotropic span, therefore any line in the closure must satisfy the same property. But we have seen that there exist lines in $X_c$ with non isotropic span.}\end{footnote} The flatness follows from \cite[Theorem 23.1]{Matsumura}, taking into account that both source and target of $\ev_1$ are smooth varieties, and that all fibers have the same dimension. Flatness implies that the GW variety $GW_1(w)$ from part (b) is pure dimensional of expected dimension. Further, using transitivity and applying \cite[Prop. 2.3]{BCMP:qkfin} to each irreducible component of $\ev_1^{-1}(X_c)$ implies that the map $\ev_1: \ev_1^{-1}(X_c) \to X_c$ is a locally trivial fibration  with fibre $F$. Then the restriction to $\ev_1^{-1}(X(w))$ is a locally trivial fibration over $X(w)$ with fibre $F$, and the statement in (b) follows. \end{proof}

\subsection{Lines with two marked points} Define $\xi: \overline{\mathcal{M}}_{0,2}(\IG,1) \longrightarrow \overline{\mathcal{M}}_{0,1}(\IG,1)$ to be the map forgetting the second marked point. \begin{prop}\label{prop:forget} The forgetful map $\xi: \Mb_{0,2}(\IG,1) \longrightarrow \Mb_{0,1}(\IG,1)$ is a locally trivial $\mathbb{P}^1$-fibration. \end{prop}
\begin{proof} Consider the embedding $\IG \subset \Gr:=\Gr(k,2n+1)$. We first prove the statement with $\IG$ replaced by $\Gr$. Recall that the moduli space
$\Mb_{0,1}(\Gr, 1) $ may be identified to the partial flag manifold $\mathrm{Fl}(k-1,k,k+1; 2n+1)$. It follows in particular that $\Mb_{0,1}(\Gr, 1)$ admits a transitive action of $\SL:=\SL_{2n+1}$. Then by \cite[Prop. 2.3]{BCMP:qkfin} the forgetful map $\overline{\mathcal{M}}_{0,2}(\Gr,1) \to \Mb_{0,1} (\Gr,1)$ is an $\SL$-equivariant locally trivial fibration with fibres isomorphic to $\mathbb{P}^1$. Consider the commutative diagram:
\[\begin{tikzpicture}[->,>=stealth',auto,node distance=3cm,main node/.style={font=\large}]


  \node[main node] (1) {$\overline{\mathcal{M}}_{0,2}(\IG,1)$};
  \node[main node]  (2) [below right of=1]{FP};
  \node[main node]  (3) [below of=2]{$\overline{\mathcal{M}}_{0,1}(\IG,1)$};
 \node[main node]  (4) [right of=2]{$\overline{\mathcal{M}}_{0,2}(\Gr,1)$};
  \node[main node]  (5) [right of=3]{$\overline{\mathcal{M}}_{0,1}(\Gr,1)$};

  \path[every node/.style={font=\sffamily\large}]
    (1) edge [dotted] node {$\psi$}  (2)
    (2) edge  node {$\pi_1$}  (3)
    (1) edge [bend right] node {$\xi$}  (3)
    (1) edge [bend left] node {$j''$}  (4)
     (2) edge node {$\pi_2$}  (4)
     (4) edge node {$\xi_2$}  (5)
     (3) edge node {$j'$}  (5);
\end{tikzpicture}\]
%
%
%
where $FP$ denotes the fibre product and $j', j''$ are the closed embeddings determined by the embedding $\IG \subset \Gr$. The map $\psi$ is determined by the universal property for fibre products. It is easy to check that $\psi$ is bijective. Since both $\Mb_{0,2}(\IG,1)$ and $FP$ are smooth varieties $\psi$ is in fact an isomorphism, by Zariski's Main Theorem. Since the right vertical arrow is a $\mathbb{P}^1$-fibration, so is the left vertical arrow $FP \simeq \Mb_{0,2}(\IG,1) \to \Mb_{0,1}(\IG,1)$. This proves the statement. \end{proof}

Combining Proposition \ref{prop:forget} and Theorem \ref{thm:fibres} imply the main result of this section. Recall that $GW_1(w)$ denotes the Gromov-Witten variety $\ev_1^{-1}(X(w))$. Obviously $\ev_1$ is the composition of the forgetful map $\xi$ with the evaluation map from $\Mb_{0,1}(\IG,1)$. 

\begin{cor}\label{cor:GW} Consider a Schubert variety $X(w) \subset X_c$. Then the Gromov-Witten variety $GW_1(w)$ has two irreducible components \[  GW_1(w) = GW_1^{(1)}(w) \cup GW_1^{(2)}(w) \/, \] where $GW_1^{(1)}(w)= \xi^{-1}(A_1)$ is the closure of the subvariety corresponding to lines $L$ such that $L \cap X^\circ \neq \emptyset$, and $GW_1^{(2)}(w)= \xi^{-1}(A_2)$ is the closed subscheme corresponding to lines $L$ included in the closed orbit $X_c$. Further, each irreducible component is generically smooth and it has dimension $\dim \Mb_{0,2}(X,1) - \mathrm{codim}_{X} X(w)$. \end{cor} 

\section{Line neighborhoods}\label{s:linenbhds} In this section we analyze the curve neighborhoods $\Gamma_1(w):=\Gamma_1(X(w))$ (i.e.~the {\em line neighborhoods}) in the case when $X(w) \subset X_c$. By Theorem \ref{thm:vanishing} these are the only ones which may contribute to non-zero GW invariants. By Corollary \ref{cor:GW}, the Gromov-Witten variety $GW_1(w)$ has two components, each of expected dimension. It follows that the curve neighborhood $\Gamma_1(w)$ has at most two components, and we have an equality \[ \Gamma_1(w)= \Gamma_1^{(1)}(w) \cup \Gamma_1^{(2)}(w) \/, \] where $\ev_2: GW_1^{(i)}(w) \to \Gamma_1^{(i)}(w):=\ev_2(GW_1^{(i)}(w))$ ($i=1,2$). By definition, $\Gamma_1^{(1)}(w) \cap X^\circ \neq \emptyset$, $\Gamma_1^{(2)}(w) \subset X_c$, and each of $\Gamma_1^{(i)}(w)$ is irreducible and stable under the standard Borel subgroup; therefore each must be a Schubert variety. Further, since the second component $GW_1^{(2)}(w)$ is the GW variety of lines in the closed orbit $X_c$ - isomorphic to the homogeneous space $\IG(k-1, 2n)$ - it follows from Corollary \ref{cor:evest} that $\Gamma_1^{(2)}(w) = X(w \cdot O_2 W_P)$ where $X(O_2) = \Gamma_1^{X_c}(id)$ is the line neighborhood of the Schubert point in $X_c$. Next we will identify the components $\Gamma_1^{(i)}(w)$. 
\begin{prop}
\label{zzzt} Consider the minimal length representatives $O_1=(2<3<\cdots<k<\overline{k+1})$ and $O_2=(1<3<4<\cdots<k<\overline{2})$. Then the line neighborhood of the Schubert point in $\IG(k,2n+1)$ is $\Gamma_1(id)=X(O_1) \cup X(O_2)$ and $\ell(O_1)=\ell(O_2)=2n+1-k$. (Observe that this equals $\deg q -1 $.) 
\end{prop}

Before the proof, we contrast the result above to that for curve neighborhoods in a homogeneous space. For the latter, it was proved in \cite{BCMP:qkfin} and \cite{buch.m:nbhds} that any curve neighborhood of a Schubert variety is a single Schubert variety. For the quasi-homogeneous space $\IG$, this already fails for $\Gamma_1(id)$, but we observe that the components correspond naturally to the orbits of $\Sp_{2n+1}$ on $\IG$.

\begin{proof} The properties of $O_2$ follow from Corollary \ref{cor:evest}. We observed in \S \ref{s:moment} that there exists a line joining $O_1$ to the identity. The fact that $\ell(O_1)= 2n+1-k$ follows immediately from the equation (\ref{E:lambdaO}) below, where we describe $O_1$ in terms of partitions. Finally, since $\ell(O_1) = \deg q -1$, Theorem \ref{Lem10} implies that $X(O_1)$ is a component of $\Gamma_1(w)$. \end{proof}

\begin{thm} \label{thm:curveNBHDS} Let $w=(w(1)<\cdots<w(k)) \in W^P \cap W^{odd}$ be an odd-symplectic minimal length representative such that  $X(w) \subset X_c$ (i.e.~$w(1)=1$). Then the cosets $w\cdot O_1W_P$ and $w \cdot O_2W_P$ have representatives in $W^{odd}$ and \[ \Gamma_1(X(w))=X(w \cdot O_1W_P) \cup X(w \cdot O_2 W_P) \/. \] \end{thm}

\begin{proof} The existence of representatives in $W^{odd}$ follows from Lemma \ref{lemma:hclosure}. To prove the equality, since both sides are $B$-stable, it suffices to check they have the same $T$-fixed points. Using the action of $\Sp_{2n+1}$, a line passing through the Schubert point $X(id)$ can be translated so it contains any point in the closed orbit. In particular, a $T$-stable line guaranteed by Proposition \ref{zzzt},  joining $X(id)$ to $O_iW_P$ ($ i = 1,2$) is translated to one joining any $T$-fixed point $v \in X(w)$ to $v O_iW_P$. Since the minimal length representatives satisfy $v \le w$ it follows that $vO_i \le v \cdot O_i\le w \cdot O_i$, therefore $\Gamma_1(X(w)) \subset X(w \cdot O_1W_P) \cup X(w \cdot O_2W_P)$. For the converse inclusion we will consider only lines $L$ which intersect both $X(w)$ and the open orbit $X^\circ$ (those included in the closed orbit are already accounted by the equality $\Gamma_1^{(2)}(w) = X(w \cdot O_2 W_P)$). Let $v=(w \cdot O_1)O_1^{-1}$, where the products are performed in $W$. 
Then $v \le w$ by \cite[Prop.3.1]{buch.m:nbhds} and $v O_1 = w \cdot O_1$ in $W$. 
 If $L$ is the line joining $X(id)$ to $O_1$ in $\IG$ then $v. L$ joins $v \in X(w)$ to $vO_1 W_P = (w \cdot O_1) W_P \in X(w \cdot O_1W_P)$. This proves the required inclusion. 
\end{proof}
%
%
%
%
\begin{figure}[h!]
\caption{The moment graph of $\IG(2,5)$. The thick edges have degree $2$ and the rest have degree $1$. Each red vertex is in $X_c$. The blue edges leaving upwards from each red vertex $w$ are connected to ${w \cdot O_1}$ or ${w \cdot O_2}$.}
\begin{tikzpicture}[-,>=stealth',shorten >=1pt,auto,node distance=1.9cm,
                    thick,main node/.style={draw,font=\sffamily\tiny\bfseries}]

  \node[main node] (1) {$(\bar{3}<\bar{2})$};
  \node[main node] (2) [below of=1] {$(3<\bar{2})$};
    \node[main node] (3) [below left of=2] {$(2<\bar{3})$};
  \node[main node] (4) [ red, below right of=2] {$(1<\bar{2})$};
  \node[main node] (5) [below of=3]{$(2<3)$};
  \node[main node] (6) [ red, below of=4]{$(1<\bar{3})$};
   \node[main node] (7) [red, below right of=5]{$(1<3)$};
   \node[main node] (8) [ red, below of=7]{$(1<2)$};

  \path[every node/.style={font=\sffamily\large}]
    (8) edge node  {} (7)
         edge node {} (5)
          edge node {} (6)
          edge [blue] node {} (3)
          edge [blue] node {} (4)
    (7) edge node {} (5)
          edge [blue] node {} (2)
          edge [blue] node {} (4) 
          edge node {} (6)
    (6) edge node [right] {} (3)
          edge  [blue] node [right] {} (1)
          edge  [blue] node [right] {} (4)
    (5) edge  node {} (3)
          edge [ultra thick] node {} (1)
          edge node {} (2)
    (4) edge node {} (2)      
          edge  [blue] node {} (1)
    (3) edge  node {} (1)
          edge [ultra thick] node [right] {} (2)
    (2) edge node {} (1);
\end{tikzpicture}
\end{figure}
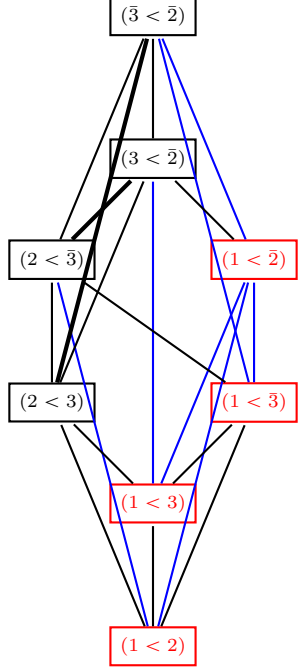

We record an immediate consequence of Lemma \ref{lemma:bound}, which gives necessary and sufficient conditions for the components of the curve neighborhood to have the expected dimension. 
\begin{lemma}\label{lemma:dim} Let $w \in W^P \cap W^{odd}$ such that $X(w) \subset X_c$, and let $z \in \{ O_1, O_2 \}$. Then $\dim X(w \cdot z W_P) - \dim X(w) \le \deg q -1$. Furthermore, the following are equivalent:

(i) $\dim X(w \cdot z W_P) - \dim X(w) = \deg q -1$;

(ii) $\ell (w \cdot z) = \ell(w) + \ell(z)$, $w \cdot z = wz$ and $w \cdot z$ is a minimal length representative in $W^{odd}$. \end{lemma}

\section{Gromov-Witten invariants of lines}\label{s:gwlines} The main result of this section is the following:

\begin{thm}\label{thm:gw1} Let $X(w) \subset X_c$ be a Schubert variety in the closed orbit of $X$, and let $z \in \{ O_1 , O_2 \}$. Consider the restricted evaluation map \[ \ev_2: GW_1^{(i)}(w) \to X(w \cdot O_i) \quad  (i=1,2) \/. \] Then $\dim GW_1^{(i)}(w) \ge \dim X(w \cdot O_i W_P)$ with equality if and only if the restricted map $\ev_2$ is birational. In particular, the following holds:
\[ (\ev_2)_*[GW_1^{(i)}(w)]_T= \begin{cases} [X(w \cdot O_i)]_T & \textrm{if } w \cdot O_i \in W^P \cap W^{odd} \textrm{and } \ell(w \cdot O_i)= \ell(w) + \ell(O_i); \\ 0 & \textrm{ otherwise}\/. \end{cases} \] 
\end{thm}
\begin{proof} By definition $X(w \cdot O_i W_P) = \ev_2(GW_1^{(i)}(w))$, therefore the inequality on dimensions is immediate. In the case of equality it remains to prove the birationality statement. First observe that in this case $\dim X(w \cdot O_i W_P) = \dim X(w W_P) + \deg q-1$, and by Lemma \ref{lemma:dim} $w \cdot O_i$ is a minimal length representative satisfying $\ell( w\cdot O_i) = \ell(w) + \ell(O_i)$. Given this, we will drop $W_P$ from the notation. 

Recall from Corollary \ref{cor:GW} that $GW_1^{(i)}(w)$ is irreducible and generically smooth. Since the evaluation map $\ev_2$ is $B$-equivariant, \cite[Prop. 2.3]{BCMP:qkfin} implies that $\ev_2$ is a locally trivial fibration over the open cell $X(w \cdot O_i)^\circ$. The preimage $\ev_2^{-1}(X(w \cdot O_i)^\circ)$, being open and dense, intersects the smooth locus of $GW_1^{(i)}(w)$. Therefore all fibres over the open cell, which by hypothesis are discrete, must be reduced. To prove birationality it suffices to show that for some $x \in X({w \cdot O_i})^{\circ}$ there exists a unique line $L$ such that $x \in L$ and $L \cap X(w) \neq \emptyset$. If $i=2$ (when the GW variety parametrizes lines included in the closed orbit) this statement follows from Lemma \ref{lemma:fibrec}. We assume from now on that $i=1$. 

We consider the fibre over $x= wO_1 = w \cdot O_1$. This fibre contains the line $L_w$, obtained by $w$-translating the unique, $T$-stable, line joining $X(id)$ and $O_1$. 
If $X(v) \subset X(w) $ such that $v \neq w$ then \[ \dim \Gamma_1(X(w)) - \dim X(v) = \ell({w\cdot O_1})- \ell(v) = 2n+1 - k + (\ell(w) - \ell(v)) >2n+1-k \/. \] Then theorems \ref{Lem10} and \ref{thm:curveNBHDS} imply that there is no line joining $X(v)$ to the open cell $X(w O_1)^\circ$. We deduce that 
any line passing through $wO_1$ and $X(w)$ cannot intersect the boundary $X(w) \setminus X(w)^\circ$ of $X(w)$. 
Let $L$ be any line such that $w O_1 \in L$ and $y \in L \cap X(w)^\circ$. If $y = w$ then $L= L_w $ is $T$-stable, so assume $y \neq w$; in particular $L$ is not $T$-stable. We show that existence of this $L$ leads to a contradiction. Consider a general $\C^* \subset T$ such that the $T$ and $\C^*$ fixed points in $\IG$ coincide. (Pick the $\C^*$ to be a regular $1$-parameter subgroup as in \cite[Ch. 24]{humphreys:linalggrps}.) 
A line $t.L$ in the (infinite) family of lines $\{t.L :t  \in \C^* \}$ contains $w \cdot O_1$ and it passes through $t.y \in X(w)^\circ$. The limits at $0$ and $\infty$ exist by the properness of the appropriate Hilbert scheme \cite[Prop. 3.9.8]{hartshorne}, and they correspond to two lines passing through two (distinct) $T$-fixed points $\lim_{t \to 0} t.y, \lim_{t \to \infty} t.y \in X(w)$. The two lines are necessarily $T$-stable, and this contradicts the uniqueness of $L_w$. 
\end{proof}

As a corollary, we can calculate the Chevalley GW invariants not covered by Theorem \ref{thm:vanishing}. 
Recall that $X(Div)$ denotes the Schubert divisor in $\IG$. 
\begin{cor} \label{cor:gw1}
Let $u,w \in W^P \cap W^{odd}$ such that $X(w) \subset X_c$. Then the Gromov-Witten invariant $\langle [X(Div)]_T, [X(w)]_T, [X(u)]_T^\vee\rangle_1 =1$ if $u=w O_i$ and $\ell(u) = \ell(w) + \ell(O_i)$, and it is equal to $0$ otherwise. \end{cor} 

\begin{proof} As in the proof of Theorem \ref{thm:vanishing} we obtain
\[ \langle [X(Div)]_T, [X(w)]_T, [X(u)]_T^\vee\rangle_1 = \int_{\IG(k, 2n+1)}^T (\ev_2)_* ( \ev_1^*[X(w)]_T) \cap [X(u)]_T^\vee \/. \] (We omitted the virtual class, since for $d=1$ this is the actual fundamental class.) By Theorem \ref{thm:fibres} and Proposition \ref{prop:forget}, the evaluation map $\ev_1$ is flat. Then by Corollary \ref{cor:GW}, \[ \ev_1^*[X(w)]_T = [\ev_1^{-1}(X(w))]_T = [GW_1^{(1)}(w)]_T + [GW_1^{(2)}(w)]_T \/. \] Then the result follows from Theorem \ref{thm:gw1} and Poincar{\'e} duality. 
\end{proof}

The previous corollary together with Theorem \ref{thm:vanishing} give the quantum terms in the equivariant quantum Chevalley formula for $X$. Recall that the Chevalley formula is given by \[[X(Div)]_T \star [X(w)]_T = \sum_{d \ge 0; u \in W^{2n+1}} c_{Div, w}^{u,d} q^d [X(u)]_T \/,\] where $c_{Div, w}^{u,d}$ is a homogeneous polynomial of degree $\mathrm{codim}~ X(w)  +1 - (\mathrm{codim} ~X(u) + d \deg q)$. The terms when $d=0$ (i.e. the non-quantum, equivariant coefficients) can be obtained from Mihai's work \cite{mihai:odd}. Those for $d >0$ are listed below. We remark that these coefficients were also calculated by Pech for $\QH^*(\IG(2, 2n+1))$ \cite{pech:quantum} and they were conjectured in few cases for $\QH^*(\IG(3, 2n+1))$ \cite{pech:thesis}. 

\begin{thm}\label{thm:eqchev} Let $u,w \in W^P \cap W^{odd}$ and $d>0$. The equivariant quantum Chevalley coefficients $c_{Div,w}^{u,d}=0$ for $d\ge 2$ or if $w(1) \neq 1$ (i.e. $X(w) \nsubseteq X_c$). If $d=1$ and $w(1)=1$ then \[ c_{Div,w}^{u,1} = \begin{cases} 1 & \textrm{if } u=w O_i \textrm{ and }\ell(u) = \ell(w) + \ell(O_i) \textrm{ for } i=1,2; \\ 0 & \textrm{otherwise} \/. \end{cases} \]
\end{thm}
In the next section we will rewrite this formula in terms of partitions.

\section{Equivariant Quantum Chevalley Rule with $(n-k)$-strict partitions}\label{s:partitions} The goal of this section is to give an explicit formulation of the equivariant quantum Chevalley formula using partitions. 

\subsection{A dictionary permutations - partitions}\label{ss:dictionary} In this section we introduce a variant of Buch, Kresch and Tamvakis $k$-strict partitions \cite{BKT2}. This variant, due to Pech \cite{pech:quantum}, is convenient to describe the cohomology of the odd-symplectic Grassmannian $X=\IG(k, 2n+1)$. Recall that if $P_k$ is the maximal parabolic subgroup of $\Sp_{2n+2}$ determined by the simple root $\alpha_k$, then the minimal length representatives $W^{P_k}$ have the form $(w(1)<w(2)<\cdots<w(k))$. Consider the set of partitions $\lambda=(2n+2-k \geq \lambda_1 \geq \cdots \geq \lambda _k \geq 0)$ which are  {\em $(n+1-k)$-strict}, i.e. $\lambda_j>\lambda_{j+1}$ whenever $\lambda_j>n+1-k$. We denote this set by $\Lambda^{2n+2}_k$. There is a bijection between $\Lambda^{2n+2}_k$ and the set $W^{P_k}$ of minimal length representatives given by:
\begin{eqnarray*}
\lambda &\mapsto& w \mbox{ is defined by }w(j)=2n+3-k-\lambda_j+\#\{ i<j : \lambda_i+\lambda_j \leq 2(n+1-k)+j-i\},\\
w &\mapsto& \lambda \mbox{ is defined by }\lambda_j=2n+3-k-w(j)+\#\{i<j: w(i)+w(j)>2n+3\} \/. 
\end{eqnarray*}
{\color{black} See \cite[Proposition 4.3]{BKT2}.} Recall that the minimal length representative of the element $w_0$ defined in (\ref{E:w0}) indexes $\IG$ as a Schubert variety inside $\IG(k, 2n+2)$. Under the bijection above, the coset of $w_0 W_{P_k}$ corresponds to the $(n+1-k)$-strict partition $1^k:= (1,1, \ldots , 1)$ if $k < n+1$ and to $(k,0,\ldots , 0 )$ if $k=n+1$. The minimal length representatives for odd symplectic permutations $w \in W^{odd}$ are in bijection with the subset of $\Lambda^{2n+2}_k$ consisting of those $(n+1-k)$-strict partitions satisfying the additional condition that if $\lambda_k = 0$ then $\lambda_1 = 2n+2 - k$; in other words, if the first column is not full, then the first row must be full.\begin{footnote} {One word of caution: the Bruhat order does not translate into partition inclusion. For example, $(2n+2-k, 0, \ldots , 0) \le (1, 1, \ldots 1) $ in the Bruhat order for $k < n+1$.}\end{footnote}  P{e}ch introduced an equivalent indexing set, which is more convenient in the context of the odd-symplectic Grassmannians: \[ \Lambda := \{ \lambda =  (2n+1-k \geq\lambda_1 \geq \cdots \geq \lambda_k \geq -1): \lambda \textrm{ is } n-k \textrm{-strict}, \textrm{ if } \lambda_k=-1 \textrm{ then } \lambda_1=2n+1-k \} \/. \] Pictorially, the partitions in $\Lambda$ are obtained by removing the full first column $1^k$ from the partitions in $\Lambda^{2n+2}_k$, regardless of whether a part equal to $0$ is present.
\begin{example}

Let $k=5$, $n=7$, and $w=(1<6<\bar{8}<\bar{7}<\bar{2}|3<4<5) \in W^{odd}$. Then $\lambda=(\lambda_1 \geq \lambda_2 \geq \lambda_3 \geq \lambda_4 \geq \lambda_5) \in \Lambda^{2n+2}$ is given by $\lambda = (11, 6, 3,3,0)$ and the corresponding partition in $\Lambda$ is $(10,5,2,2,-1)$.
Pictorially, \[ \tableau{6}{&{}&{}&{}&{}&{}&{}&{}&{}&{}&{}&{}\\&{}&{}&{}&{}&{}&{}\\&{}&{}&{}\\ &{}&{}&{}}   - \tableau{6}{{}&\\ {}&\\ {}&\\ {}&\\ {}} = \tableau{6}{&{}&{}&{}&{}&{}&{}&{}&{}&{}&{}\\&{}&{}&{}&{}&{}\\&{}&{}\\ &{}&{} \\ {}&&&&& && } \]   
%
\end{example}
\begin{example} Let $k= n+1 =5$, so $\IG(5,9) \simeq \IG(4,8)$ is the Lagrangian Grassmannian. Then the codimension $0$ class is the $-1$-strict partition $\lambda = (4,-1,-1,-1,-1)= \tableau{5}{&{}&{}&{}&{}\\{}\\{}\\{}\\{}} $.\end{example}

For $\lambda \in \Lambda$ define $|\lambda|=\lambda_1+\ldots+\lambda_k$. If $w$ corresponds to $\lambda$ then $\ell(w)=k(2n+1-k)-\frac{k(k-1)}{2}-|\lambda|$, i.e. the codimension of the Schubert variety $X(w)$ in $X$ equals $|\lambda|$; see \cite[Proposition 4.4]{BKT2} and \cite[Section 1.1.1]{pech:quantum}. The partitions associated to the elements $O_1$ and $O_2$ from Proposition \ref{zzzt} are: \begin{equation}\label{E:lambdaO} \begin{split} \lambda(O_1) = (2n-k \geq 2n- k- 1 \geq 2n-k -2 \geq ...\geq 2n-2k+ 2 \geq 0) \/; \\ \lambda(O_2) = (2n-k +1 \geq 2n- k -1 \geq ... \geq 2n-2k+2  \geq -1) \/ . \end{split}
\end{equation} A Schubert variety $X(w)$ is included in the closed orbit $X_c \subset \IG$ if its partition $\lambda \in \Lambda$ satisfies $\lambda_1 = 2n+1-k$. In order to translate the conditions from Lemma \ref{lemma:dim} in terms of partitions we need the following definition.
\begin{defn}\label{def:lambda*} Let $\lambda = (\lambda_1, \ldots , \lambda_k)$ be a partition in $\Lambda$ such that $\lambda_1 = 2n+1-k$.

(a) If $\lambda_k \geq 0$ then let $\lambda^{*}=(\lambda_2 \geq \lambda_3 \geq \cdots \geq \lambda_k \geq 0)$. If $\lambda_k=-1$ then $\lambda^{*}$ does not exist.

(b) If $\lambda_2=2n-k$ then let $\lambda^{**}=(\lambda_1 \geq \lambda_3 \geq \cdots \geq \lambda_k \geq -1)$. If $\lambda_2<2n-k$ then $\lambda^{**}$ does not exist. 
\end{defn}
In both situations notice that $|\lambda^*| = |\lambda^{**}| =  |\lambda| -( 2n+1 - k)$. As an example, if $\rho = (2n-k+1, 2n-k, \ldots , 2n-2k+2)$ is the partition indexing the Schubert point, then 
$\lambda(O_1) = \rho^{*}$ and $\lambda(O_2) = \rho^{**}$. It is easy to produce examples when only one of $\lambda^*$ or $\lambda^{**}$ exist. For instance, if $k=3,n=4$, and $\lambda = (6,5, -1)$ then $\lambda^*$ does not exist, but $\lambda^{**} = (6, -1, -1)$; if $\lambda = (6,3,0)$ then $\lambda^* = (3,0,0)$ and $\lambda^{**}$ does not exist.

%
%
%
%

\begin{prop}\label{prop:dictionary} Let $w \in W^{P_k} \cap W^{odd}$ such that $w(1) =1$ and let $w \mapsto \lambda= (2n+1-k, \lambda_2, \ldots , \lambda_k)$ be the partition in $\Lambda$ corresponding to $w$. The following hold:

(a) The partition $\lambda^*$ exists if and only if $wO_1$ is a minimal length representative in $W^{odd}$. If any of these conditions is satisfied then $w O_1 \mapsto \lambda^{*}$, thus in particular $\ell(w O_1) = \ell(w)+ \ell(O_1)$.

(b) The partition $\lambda^{**}$ exists if and only if $wO_2$ is a minimal length representative in $W^{odd}$ and $\ell(w O_2) = \ell(w)+ \ell(O_2)$. In this case $w O_2 \mapsto \lambda^{**}$.
\end{prop}
\begin{proof} By Lemma \ref{lemma:hclosure} $w O_i \in W^{odd}$ so one only needs to check the claims about minimal length representatives. Let $w= (1 < 2 < \ldots < j < w(j+1) < \ldots < w(k) | w(k+1) < \ldots < w(n+1))$ where $j \geq 1$, $w(j+1) > j+1$ and $w(n+1) \leq n+1$, since $w$ is a minimal length representative; this last condition is omitted if $k = n+1$. Notice that either $w(k) = \overline{j+1}$ and $j+2 \leq w(k+1)$, or $w(k) \leq \overline{j+2}$ and $w(k+1) = j+1$. By the definition of $O_1$ and $O_2$, we have \[ \begin{split} wO_1 = (w(2), w(3), \ldots ,  w(k), \overline{w(k+1)}|1, w(k+2),  \ldots , w(n+1)) \/; \\ wO_2 = (1, w(3), \cdots , w(k), \overline{w(2)}|w(k+1),  \cdots w(n+1)) \/,  \end{split} \] as elements in $W$. Therefore $wO_1$ is not a minimal length representative if and only if $w(k) > \overline{w(k+1)}$, i.e. $w(k) = \overline{j+1}$ and $j+2 \leq w(k+1)$. Similarly, $wO_2 \notin W^{P_k}$ if and only if $w(k) > \overline{w(2)}$. 

We now proceed to prove the statement (a). If $\lambda^*$ exists but $w(k+1) \neq j+1$, then the preceding considerations imply that $w(k) = \overline{j+1}$. Using the bijection $W^{P_k} \cap W^{odd} \to \Lambda$, we calculate
\begin{eqnarray*}
\lambda_k&=&2n+2-k-\overline{j+1}+\#\{i<k:w(i)+\overline{j+1}>2n+3 \}\\
&=&j-k+\#\{i<k:w(i)>j+1 \}\\
&=&j-k+(k-1-j)\\
&=&-1.
\end{eqnarray*}
This contradicts that $\lambda_k \geq 0$. Therefore $w(k) \leq \overline{j+2}$ and $w(k+1) = j+1$, which means that $wO_1 \in W^{P_k}$. Conversely, if $wO_1$ is a minimal length representative, let $wO_1 \mapsto \mu$ under the bijection $W^{P_k} \cap W^{odd} \to \Lambda$. Then for $1 \leq s \leq k-1$, \[ \begin{split} \mu_s &=  2n+2-k-wO_1(s)+ \# \{ i <s: wO_1(i) + wO_1(s) > 2n+3 \} \\ &= 2n+2 - k - w(s+1) + \# \{ i <s: w(i+1) + w(s+1) > 2n+3 \} \/. \end{split} \] Since $w(1) =1$, $ \# \{ i <s: w(i+1) + w(s+1) > 2n+3 \} =  \# \{ i <s+1: w(i) + w(s+1) > 2n+3 \}$, thus $\mu_s = \lambda_{s+1}$. We calculate $\mu_k$ separately: \begin{eqnarray*}
\mu_k&=&2n+2-k-wO_1(k)+\#\{i<k:wO_1(i)+wO_1(k)>2n+3 \}\\
& = & 2n+2-\overline{j+1}+\#\{i<k:w(i+1)+\overline{j+1}>2n+3 \} \\
&=&j-k+\#\{i<k:w(i+1)>j+1 \}\\
&=&j-k+(k-j)\\
&=&0.
\end{eqnarray*} Then $\mu = \lambda^*$, and in particular the length condition is satisfied.  

We now prove (b). We first observe that $\lambda_2 = 2n+2 - k - w(2)$. Then $\lambda^{**}$ exists if and only if $\lambda_2 = 2n-k$, i.e. $w(2) = 2$. Then clearly $w(k)< \overline{w(2)}$, therefore $wO_2 \in W^{P_k}$. Let $wO_2 \mapsto \mu$. As before we calculate $\mu_1 = 2n+1-k$, $\mu_{s} = \lambda_{s+1}$ for $2 \leq s \leq k-1$, and that $\mu_k=w(2)-3 = -1$. This proves one implication. For the converse, we notice that once $wO_2 \in W^{P_k}$, same calculations show that $\mu_1 = 2n+1-k$ and that $\mu_s = \lambda_{s+1}$ for $2 \leq s \leq k-1$ (the condition $w(2) =2$ is not used in these). The length condition on $\ell(wO_2)$ implies that $|\lambda| - |\mu| = 2n+1-k$, which forces $\mu_k=-1$, thus $\mu = \lambda^{**}$ and the proposition is proved.\end{proof}

\subsection{The equivariant quantum Chevalley formula} To formulate the equivariant quantum Chevalley formula we will first recall the (non-quantum) equivariant Chevalley formula for $\IG$. This is due to P{e}ch \cite{pech:thesis}, but for the convenience of the reader we briefly recall the main steps. (P{e}ch works in the non-equivariant setting, and a minor argument is needed for the equivariant extension.) In a nutshell, P{e}ch uses the embedding $\iota: \IG \to \IG(k, 2n+2)$ to reduce the calculation to the Chevalley formula in $H^*(\IG(k, 2n+2))$. Since $\IG(k, 2n+2)$ is homogeneous, the classical work of Chevalley \cite{Chevalley}, and its equivariant generalization (see e.g. \cite{KK}), give this formula with Schubert classes indexed by Weyl group representatives. Buch, Kresch and Tamvakis \cite{BKT2} proved a more general (non-equivariant) Pieri rule, and in the process re-stated the formulas in terms of strict partitions.  

In this section we will use the notation $X(\lambda)$ to denote the Schubert variety in $\IG$ and $Y(\lambda+1)$ to denote the same Schubert variety, but now regarded in $\IG(k, 2n+2)= \Sp_{2n+2}/P$. The notation is consistent with the fact that the partitions in $\Lambda^{2n+2}_k$ are obtained from the ``odd-symplectic partitions" $\lambda \in \Lambda$ by adding one box to each row. Set $X(1)$ respectively $Y(1)$ to be the Schubert divisors in $\IG$ and in $\IG(k, 2n+2)$. P{e}ch proved that in $H^*(\IG)$ there is an equality $\iota^*[Y(1)] = [X(1)]$. Therefore in the equivariant cohomology $\iota^*[Y(1)]_T = [X(1)]_T + C(t)$, where $C(t)\in H^2_T(pt)$ is a homogeneous linear form. After localization at the point $w_0W_P$ (the torus-fixed point in the open Schubert cell in $\IG$), and using that $w_0 W_P \notin X(1)$ we obtain that $C(t) = \iota_{w_0}^*[Y(1)]_{T}$, where $\iota_{w_0}^*$ is the localization map. For the next result, let $\overline{w_0}$ denote the longest element in $W$.
\begin{lemma}\label{lemma:loc} Let $w \in W$ be a signed permutation. Then the localization coefficient $\iota_w^* [Y(1)]_T = \overline{w_0} (\omega_k) - w (\omega_k)$. In particular, $C(t)=\iota_{w_0}^*[Y(1)]_{T}$ equals \[ C(t) = \begin{cases} t_{k+1} - t_1 & \textrm{ if } k < n+1; \\ - 2 t_1 & \textrm{ if } k = n+1 \/. \end{cases} \] \end{lemma} 
\begin{proof} Let $\varphi_{\overline{w_0}}: \IG(k, 2n+2) \to \IG(k, 2n+2)$ be the left multiplication by $\overline{w_0}$. This is an automorphism of $\IG(k, 2n+2)$ which is equivariant with respect to the map $T \to T$ given by $ t \mapsto \overline{w_0} t \overline{w_0}^{-1}$. There is a commutative diagram \[ \xymatrix{ \IG(k, 2n+2) \ar[r]^{\varphi_{\overline{w_0}}} & \IG(k, 2n+2) \\ \{ w \} \ar[u]^{\iota_{w}} \ar[r]^{\varphi_{\overline{w_0}}} & \{ \overline{w_0} w \} \ar[u]^{\iota_{\overline{w_0} w}} }  \]

For $v \in W^P$ let $\overline{Y}(v)$  denote the Schubert variety which is stable under $B_{2n+2}^-$, the opposite Borel subgroup $B_{2n+2}$; then $\overline{Y}(v)$ has codimension $\ell(v)$. The morphism $\varphi_{\overline{w_0}}$ induces a ring isomorphism $\varphi_{\overline{w_0}}^*\in Aut( H^*_{T}(\IG(k, 2n+2)))$ which satisfies $\varphi^*_{\overline{w_0}}[\overline{Y}(\overline{w_0} v W_P)]_{T} = [Y(v W_P)]_{T}$, and it acts on  $H^*_{T}(pt)$ by twisting by $\overline{w_0}$. Since in our situation $[Y(1)]_{T} = \varphi^*_{\overline{w_0}}[\overline{Y}(s_k)]_{T}$ we deduce that \[ \begin{split}  \iota_{w}^*[Y(1)]_T =  \iota_{w}^* \varphi^*_{\overline{w_0}} [\overline{Y}( s_k)]_{T} = \varphi^*_{\overline{w_0}} \iota^*_{\overline{w_0}w} [\overline{Y}( s_k)]_{T}  =  \varphi^*_{\overline{w_0}} (\omega_k - \overline{w_0}w(\omega_k)) = \overline{w_0} (\omega_k) - w (\omega_k)  \/. \end{split} \] 
The third equality follows from localization formulas of Schubert classes for opposite Borel subgroups, see e.g.~\cite{KK}. The claim on $C(t)$ follows from taking into account the expression for $w_0$ from (\ref{E:w0}), that $\omega_k = t_1 + \ldots + t_k$, and that $\overline{w_0} = (\bar{1} , \ldots , \overline{n+1})$.\end{proof}

%
%

Consider the expansions \begin{equation}\label{E:expansion} \begin{split} [X(1)]_T \cup [X(\lambda)]_T & = \sum_{\mu \in \Lambda} c_{(1), \lambda}^\mu [X(\mu)]_T \in H^*_T(X) \/; \\ [Y(1)]_T \cup [Y(\lambda+1)]_T & = \sum_{\mu \in \Lambda} \tilde{c}_{(1), \lambda}^\mu [Y(\mu+1)]_T \in H^*_{T_{2n+2}}(\IG(k, 2n+2)) \/, \end{split} \end{equation} where $c_{(1), \lambda}^\mu, \tilde{c}_{(1), \lambda}^\mu \in H^*_T(pt)$. Notice that $\iota_*[X(\lambda)]_T = [Y(\lambda+1)]_T$ therefore the product $[Y(1)]_T \cup [Y(\lambda+1)]_T$ will only contain cohomology classes supported on $\IG$. We apply $\iota_*$ to both sides of (\ref{E:expansion}) and the projection formula to obtain \[ \begin{split} \iota_* ( [X(1)]_T \cup [X(\lambda)]_T) = &  \iota_*( (\iota^*[Y(1)]_T - C(t)) \cup [X(\lambda)]_T) \\ = & ([Y(1)]_T  - C(t)) \cup [Y(\lambda+1)]_T \\ = & [Y(1)]_T \cup [Y(\lambda+1)]_T - C(t) [Y(\lambda+1)]_T \end{split} \] It follows from this and Lemma \ref{lemma:loc} that {\color{black} \begin{equation}\label{E:LR} c_{(1), \lambda}^\mu = \begin{cases} \tilde{c}_{(1), \lambda}^\mu & \textrm{ if } \lambda \neq \mu \/; \\ \tilde{c}_{(1), \lambda}^\lambda -C(t)  = w_0 (\omega_k) - w_\lambda (\omega_k) & \textrm{ if } \lambda= \mu \/, \end{cases} \end{equation}}where $w_\lambda \in W$ is any permutation such that $w W_P$ corresponds to $\lambda$. Notice in particular that if $\lambda \neq \mu$, the coefficients $\tilde{c}_{(1), \lambda}^\mu$ are non-negative integers. We recall next the formula for these integers obtained in \cite{BKT2}.
\begin{defn} \label{DEFN:Related} Represent $\lambda \in \Lambda^{2n+2}_k$ as a Young diagram. The box in row $r$ and column $c$ of $\lambda$ is $(n+1-k)${\em ~-~ related} to the box in row $r'$ and column $c'$ if \[|c-n+k-2|+r=|c'-n+k-2|+r'. \] 
Given $\lambda, \mu \in \Lambda^{2n+2}_k$ with $\lambda \subset \mu$, the skew diagram $\mu / \lambda$ is called a horizontal strip (resp. vertical) strip if it does not contain two boxes in the same column (resp. row).

Following \cite[Definition 1.3]{BKT2} we say $ \lambda \rightarrow \mu$ for any $n+1-k$-strict partitions $\lambda, \mu$ if $\mu$ can be obtained by removing a vertical strip from the first $n+1-k$ columns of $\lambda$ and adding a horizontal strip to the result, so that
\begin{enumerate}
\item if one of the first $n+1-k$ columns of $\mu$ has the same number of boxes as the same column of $\lambda$, then the bottom box of this column is $n+1-k$-related to at most one box of $\mu \backslash \lambda$; and
\item if a column of $\mu$ has fewer boxes than the same column of $\lambda$, the removed boxes and the bottom box of $\mu$ in this column must each be $n+1-k$-related to exactly one box of $\mu \backslash \lambda$, and these boxes of $\mu \backslash \lambda$ must all lie in the same row.
\end{enumerate}
If $\lambda \rightarrow \mu$, we let $\mathbb{A}$ be the set of boxes of $\mu \backslash \lambda$ in columns $n+2-k$ through $2n+1-k$ which are not mentioned in (1) or (2). Then define $N(\lambda, \mu)$ to be the number of connected components of $\mathbb{A}$ which do not have a box in column $n+2-k$. Here two boxes are connected if they share at least a vertex.
\end{defn}
We refer to \cite{BKT2} for examples of these coefficients.
Combining theorem \ref{thm:eqchev}, proposition \ref{prop:dictionary}, and equation (\ref{E:LR}) above, together with the formulation of the Chevalley rule for $H^*(\IG(k, 2n+2))$ obtained in \cite[Theorem 1.1]{BKT2} yields the equivariant quantum Chevalley formula. To shorten notation we set $A(\lambda, \mu):= N(\lambda+1, \mu+1)$. 
\begin{thm}\label{thm:eqqchev} Let $\lambda \in \Lambda$ be an $n-k$ strict partition. Then the following equality holds in the equivariant quantum cohomology ring $\QH^*_T(\IG(k, 2n+1))$:
\begin{eqnarray}\label{E:eqqchev}\begin{split} 
[X(1)]_T \star [X(\lambda)]_T & = \left( \sum 2^{A(\lambda,\mu)} [X(\mu)]_T \right) +\left( w_0(\omega_k) - w_\lambda(\omega_k) \right) [X(\lambda)]_T \\ & \quad +  q[X(\lambda^*)]_T+q[X(\lambda^{**})]_T \/, \end{split} \end{eqnarray}
where the first sum is over partitions $\mu \in \Lambda$ such that $\lambda +1  \to \mu +1$ and $|\mu|= |\lambda|+1$, and where $w_\lambda (j)=2n+2-k-\lambda_j+\# \{i<j: \lambda_i+\lambda_j \leq 2(n-k)+j-i\}$.
When $\lambda^*$ or $\lambda^{**}$ do not exist then the corresponding quantum term is omitted.
\end{thm}
\begin{example} Consider the Schubert class indexed by $[X(6,2,1)]_T \in \QH^*_T(\IG(5,11))$.
The permutation corresponding to $\lambda = (6,2,1)$ is $w = (1,5,6,\bar{4},\bar{3}|2)$. Then 
\[[X(1)]_T \star [X(6,2,1)]_T=-(t_1+t_2+2t_5+2t_6) [X(6,2,1)]_T+2[X(6,3,1)]_T+q[X(2,1)]_T. \]
More examples can be found in section \ref{s:examples}.
\end{example}
\begin{remark}\label{rmk:pos} By Kleiman-Bertini theorem, the GW invariants for homogeneous spaces are enumerative; cf.~\cite{FP}.~There is an equivariant version of positivity \cite{graham:positivity, mihalcea:positivity} which states that (quantum) equivariant multiplication of $B-$stable Schubert classes yields structure constants which are polynomials in positive simple roots with (weakly) {\em negative} coefficients. Both the ordinary and equivariant positivity statements hold for the coefficients of the Chevalley formula (\ref{E:eqqchev}). Since $\IG$ is not a homogeneous space, one expects that in general positivity will fail. Based on theorem \ref{thm:alg} below we calculated that in $\QH^*_T(\IG(2,5))$, 
\[ [X(3,-1)]_T \star [X(3,-1)]_T=(t_1^2-t_3^2)[X(3,-1)]_T+(t_2+t_3)[X(3)]_T +[X(3,1)]_T -q \] thus 
$c_{(3,-1),(3,-1)}^{(0),1} = -1$ (this coefficient was also calculated by Pech \cite{pech:thesis}) and $c_{(3,-1),(3,-1)}^{(3),0} = t_2 + t_3$. The last coefficient fails the expected equivariant positivity. The full multiplication table in $\QH^*_T(\IG(2,5)$, containing more such examples, can be found in section \ref{ss:qht25}.
\end{remark}

%

\section{Application: an algorithm for the structure constants of $\QH_T^*(\IG(k,2n+1))$}\label{s:alg} One of the main applications of the equivariant quantum Chevalley formula is a recursive algorithm calculating the structure constants in the equivariant quantum cohomology ring $\QH^*_T(\IG)$. This is possible despite the fact that the divisor class does not generate the ring.\begin{footnote}{But the Schubert divisor generates the ring $\QH^*_T(\IG)$ {\em localized} at the equivariant parameters. We refer to \cite[\S 5]{BCMP:eqkt} for details.}\end{footnote} The key is that the extra equivariant parameters introduce sufficient rigidity to allow for a recursive formula. Similar algorithms, in various levels of generality, were obtained in \cite{okounkov,molev.sagan,knutson.tao} in relation to equivariant cohomology of Grassmannians. These were generalized for equivariant quantum cohomology and equivariant quantum K theory of flag manifolds in \cite{mihalcea:eqschub,mihalcea:eqqhom,BCMP:eqkt}. Although the odd-symplectic Grassmannian is not homogeneous, the shape of the equivariant quantum Chevalley formula is almost identical to the one for the Grassmannian. In particular, the non-quantum terms are governed by the Bruhat order, there are no ``mixed terms" (i.e. no terms which contain both equivariant and quantum coefficients), and there are two quantum terms with coefficient $1$ (in the Grassmannian case, there is just one such term). Therefore it should not be a surprise that almost the same algorithm as the one from \cite{mihalcea:eqschub} extends to this case, with essentially the same proof. We present next the precise results, while indicating the salient points in their proofs, but we shall leave it to the reader to check the details. 

We need to introduce few additional notations. For a partition $\lambda \in \Lambda$ there is at most one partition $\lambda^+ \in \Lambda$ such that $(\lambda^+)^* = \lambda$. Similarly, there exists at most one partition $\lambda^{++}$ such that $(\lambda^{++})^{**} = \lambda$.

\begin{prop}\label{prop:recursion} Let $\lambda, \mu, \nu \in \Lambda$ and $d \in H_2(X)$ a non-negative degree. The structure constant $c_{\lambda,\mu}^{\nu,d}$ satisfy the following equation:
\begin{equation}\label{E:rec} \begin{split} (w_\nu(\omega_k) - w_\lambda(\omega_k))c_{\lambda,\mu}^{\nu,d} & = \sum_\eta 2^{A(\lambda, \eta)} c_{\eta, \mu}^{\nu,d} - \sum_\xi 2^{A(\xi, \nu)}c_{\lambda, \mu}^{\xi,d} \\ & \quad  + (c_{\lambda^*, \mu}^{\nu, d-1} - c_{\lambda, \mu}^{\nu^+, d-1}) + (c_{\lambda^{**}, \mu}^{\nu, d-1} - c_{\lambda, \mu}^{\nu^{++}, d-1}) \/, \end{split} \end{equation} where $w_\lambda \in W$ is the partition corresponding to $\lambda$, the first sum is over $\eta \in \Lambda$ such that $\lambda+1 \to \eta +1$ and $|\eta| = |\lambda| + 1$, and the second sum is over $\xi \in \Lambda$ such that $\xi+1 \to \nu+1$ and $|\xi| = |\nu| - 1$; the terms involving $\lambda^*, \lambda^{**}, \nu^+, \nu^{++}$ are omitted if the corresponding partition does not exist. \end{prop}

\begin{proof} This is an immediate calculation obtained by collecting the coefficient of $q^d [X(\nu)]_T$ in both sides of the associativity equation $[X(1)]_T \star ([X(\lambda)]_T \star [X(\mu)]_T) = ([X(1)]_T \star [X(\lambda)]_T) \star [X(\mu)]_T$.\end{proof} 

The system of equations (\ref{E:rec}) gives a recursive procedure to calculate any structure constant $c_{\lambda,\mu}^{\nu,d}$. We briefly recall the main ideas, following \cite{mihalcea:eqschub}, where a similar equation appeared in the study of the equivariant quantum cohomology of Grassmannians; see \cite{mihalcea:eqqhom, BCMP:eqkt} for more general algorithms. The procedure can be summarized as follows: given $\lambda, \mu, \nu \in \Lambda$ and $d$ a degree, the first sum contains coefficients $c_{\eta, \mu}^{\nu,d}$ where $\eta$ is smaller in Bruhat order than $\lambda$; the second sum contains coefficients $c_{\lambda, \mu}^{\xi,d}$ where $\xi$ is larger than $\nu$ in Bruhat order; the remaining terms involve degree $d-1 < d$, known inductively. Given this, the recursion can be run whenever $w_\nu(\omega_k) - w_\lambda(\omega_k) \neq 0$, which is equivalent to asking that $\lambda \neq \nu$. If $\lambda = \nu$ one runs the recursion for the coefficient $c_{\mu,\lambda}^{\nu,d} = c_{\lambda, \mu}^{\nu,d}$, using commutativity of the quantum ring. If $\lambda = \mu = \nu$ one uses the system of equations (\ref{E:rec}) to write down a linear equation in the unknown coefficient $c_{\lambda, \lambda}^{\lambda,d}$ where all other terms in this equation will be known recursively; see \cite[Prop. 6.2]{mihalcea:eqschub}, and also \cite[Prop. 7.4]{mihalcea:eqqhom} or \cite[Prop.5.4]{BCMP:eqkt} for similar statements. The existence of the linear equation in $c_{\lambda, \lambda}^{\lambda,d}$ requires that the linear form $F_{\lambda, \nu}:= w_\lambda(\omega_k) - w_\nu(\omega_k)$ is a nonzero, positive, combination of simple roots whenever $w_\lambda < w_\nu$ in Bruhat ordering. This follows easily by induction on the length $\ell(w_\nu) - \ell(w_\lambda)$. Alternatively, $F_{\lambda, \nu} = \tilde{c}_{\nu, (1)}^\nu - \tilde{c}_{\lambda, (1)}^\lambda$, and the required positivity follows from \cite[Appendix]{mihalcea:eqqhom}, applied to $G/P= \IG(k, 2n+2)$. This proves the following:

\begin{thm}\label{thm:alg}
The EQ coefficients are determined (algorithmically) by the following formulas:
\begin{enumerate}
\item $c_{(0),(0)}^{(0),d}=\begin{cases} 0 & d>0 ~\/; \\ 1 & d=0~ \/; \end{cases}$
\item (commutativity) $c_{\lambda,\zeta}^{\mu,d}=c_{\zeta,\lambda}^{\mu,d}$ for al partitions $\lambda,\zeta,$ and $\mu$; 
\item (EQ Chevalley) The coefficients $c_{(1),\lambda}^{\mu,d}$ from theorem \ref{thm:eqqchev}, for all partitions $\lambda$ and $\mu$, and all degrees $d$;
\item The system of equations (\ref{E:rec}) for all partitions $\lambda,\zeta,\mu$ such that $\lambda \neq \nu$.
\end{enumerate}
\end{thm}
The theorem immediately implies Corollary \ref{cor:alg}, stated in the introduction. 

\section{Examples}\label{s:examples} 
In this section we present several examples. All multiplications are in the equivariant quantum cohomology ring, but we will ignore the subscripts $T$. The Chevalley formula for $\QH_T^*(\IG(3,7))$ is: 
\begin{center} 
 \begin{tabular}{|c | c|}
\hline $\lambda$ & $[X(1)]\star [X(\lambda)]$ \\
\hline $(1)$ & $ -(2 t_4)[X(1)]+[X(4,-1,-1)]+2[X(2)] $ \\
\hline $ (2)$ & $ -(2t_3)[X(2)]+[X(4,0,-1)]+[X(2,1)]+2[X(3)]$ \\
\hline $ (4,0,-1) $ & $-(t_1+t_3)[X(4,0,-1)]+[X(4,1,-1)]+[X(4)] $ \\
\hline $(2,1)$ &  $ -2(t_3+t_4)[X(2,1)]+[X(4,1,-1)]+2[X(3,1)]$ \\
\hline $(3)$ & $-(2t_2)[X(3)]+[X(3,1)]+[X(4)]$ \\
\hline $(4,1,-1) $ & $ -(t_1+t_3+2t_4)[X(4,1,-1)]+[X(4,2,-1)]+[X(4,1)] $ \\
\hline $(3,1)$ & $-2(t_2+t_4)[X(3,1)]+[X(4,1)]+2[X(3,2)]$ \\
\hline $(4)$ & $ -(t_1+t_2)[X(4)]+[X(4,1)]+q $ \\
\hline $(4,2,-1)$ & $ -(t_1+2t_3+t_4)[X(4,2,-1)]+[X(4,2)]+2[X(4,3,-1)] $ \\
\hline $(3,2)$  &  $ -2(t_2+t_3)[X(3,2)]+[X(3,2,1)]+[X(4,2)] $ \\
\hline $(4,1)$ &  $-(t_1+t_2+2t_4)[X(4,1)]+[X(4,3,-1)]+2[X(4,2)]+q[X(1)]$ \\
\hline $(4,3,-1)$ & $-(t_1+2t_2+t_4)[X(4,3,-1)]+[X(4,3)]+q[X(4,-1,-1)] $ \\
\hline $(3,2,1)$ &  $ -2(t_2+t_3+t_4)[X(3,2,1)]+[X(4,2,1)]$ \\
\hline $ (4,2) $ &  $-(t_1+t_2+2t_3)[X(4,2)]+[X(4,2,1)]+[X(4,3)]+q[X(2)]$ \\
\hline $ (4,2,1) $ &  $-(t_1+t_2+2t_3+2t_4)[X(4,2,1)]+[X(4,3,1)]+q[X(2,1)] $ \\
\hline $(4,3)$ & $ -(t_1+2t_2+t_3)[X(4,3)]+[X(4,3,1)]+q[X(4,0,-1)]+q[X(3)]$ \\
\hline $ (4,3,1)$ &  $ -(t_1+2t_2+t_3+2t_4)[X(4,3,1)]+[X(4,3,2)]
+ q[X(4,1,-1)]+q[X(3,1)]$ \\
\hline $(4,3,2)$ & $ -(t_1+2t_2+2t_3+t_4)[X(4,3,2)]+q[X(4,2,-1)]+q[X(3,2)]$ \\
\hline
 \end{tabular}

\end{center}

\subsection{Multiplication table for $\QH_T^*(\IG(2,5))$}\label{ss:qht25}
\begin{small}
\begin{eqnarray*}
\mbox{}[X(1)]&\star&[X(1)]=-2t_3[X(1)]+[X(3,-1)]+2[X(2)]\\
\mbox{}[X(1)]&\star&[X(2)]=-2t_2[X(2)]+[X(2,1)]+[X(3)]\\
\mbox{}[X(1)]&\star&[X(3,-1)]=-(t_1+t_3)[X(3,-1)]+[X(3)]\\
\mbox{}[X(1)]&\star&[X(2,1)]=-2(t_2+t_3)[X(2,1)]+[X(3,1)]\\
\mbox{}[X(1)]&\star&[X(3)]=-(t_1+t_2)[X(3)]+[X(3,1)]+q\\
\mbox{}[X(1)]&\star&[X(3,1)]=-(t_1+t_2+2t_3)[X(3,1)]+[X(3,2)]+q[X(1)]\\
\mbox{}[X(1)]&\star&[X(3,2)]=-(t_1+2t_2+t_3)[X(3,2)]+q [X(2)]+q[X(3,-1)] \\
\hline\\
\mbox{}[X(2)]&\star&[X(2)]=2t_2(t_2-t_3)[X(2)]-2t_2[X(2,1)]-(t_2-t_3)[X(3)]+[X(3,1)] \\ 
\mbox{}[X(2)]&\star&[X(3,-1)]=-(t_1+t_2)[X(3)]+q \\
\mbox{}[X(2)]&\star&[X(2,1)]=2t_2(t_2+t_3)[X(2,1)]-(t_2+t_3)[X(3,1)]+[X(3,2)]\\
\mbox{}[X(2)]&\star&[X(3)]=(t_1+t_2)(t_2-t_3)[X(3)]-(t_1+t_2)[X(3,1)]-(t_2-t_3)q+q[X(1)]\\
\mbox{}[X(2)]&\star&[X(3,1)]=(t_1 + t_2) (t_2+t_3)[X(3,1)]-(t_1+t_2)[X(3,2)]-(t_2+t_3)q[X(1)]\\
&+&q[X(3,-1)]+q[X(2)]\\
\mbox{}[X(2)]&\star&[X(3,2)]=2t_2(t_1+t_2)[X(3,2)]-(t_1+t_2)q[X(3,-1)]-2t_2q[X(2)]+q[X(3)]\\
\hline\\
\mbox{}[X(3,-1)]&\star&[X(3,-1)]=(t_1^2-t_3^2)[X(3,-1)]+(t_2+t_3)[X(3)]+[X(3,1)]-q\\
\mbox{}[X(3,-1)]&\star&[X(2,1)]=-(t_1+t_2)[X(3,1)]-[X(3,2)]+q[X(1)]\\
\mbox{}[X(3,-1)]&\star&[X(3)]=(t_1^2-t_2^2)[X(3)]+[X(3,2)]-(t_1-t_2)q\\
\mbox{}[X(3,-1)]&\star&[X(3,1)]=(t_1^2-t_2^2)[X(3,1)]-(t_2+t_3)[X(3,2)]-(t_1-t_2)q[X(1)]+q[X(2)]\\
\mbox{}[X(3,-1)]&\star&[X(3,2)]=(t_1^2-t_3^2)[X(3,2)]-(t_1-t_3)q[X(2)]+q[X(2,1)]\\
\hline\\
\mbox{}[X(2,1)]&\star&[X(2,1)]=-4t_2t_3(t_2+t_3)[X(2,1)]+2t_3(t_2+t_3)[X(3,1)]-2(t_2+t_3)[X(3,2)]\\
&+&q[X(3,-1)]\\
\mbox{}[X(2,1)]&\star&[X(3)]=(t_1+t_2)(t_2+t_3)[X(3,1)]-(t_1-t_2)[X(3,2)]-(t_2+t_3)q[X(1)] + q[X(2)]\\
\mbox{}[X(2,1)]&\star&[X(3,1)]=-2t_3(t_1+t_2)(t_2+t_3)[X(3,1)]+2t_1(t_2+t_3)[X(3,2)]\\
&-&(t_1+t_3)q[X(3,-1)]+2t_3(t_2+t_3)q[X(1)]-2(t_2+t_3)q[X(2)]+q[X(3)]\\
\mbox{}[X(2,1)]&\star&[X(3,2)]=-2t_2(t_1+t_3)(t_2+t_3) [X(3,2)]+(t_2+t_3)(t_1+t_3)q[X(3,-1)]\\
&+&2t_2(t_2+t_3)q[X(2)]-(t_1+2t_2+t_3)q[X(3)]+q^2\\
\hline\\
\mbox{}[X(3)]&\star&[X(3)]=-(t_1+t_2)(t_1t_2-t_1t_3-t_2^2+t_2t_3)[X(3)]+(t_1^2-t_2^2)[X(3,1)]-2t_2[X(3,2)]\\
&+&(t_1t_2-t_1t_3-t_2^2+t_2t_3)q-(t_1-t_2)q[X(1)]+q[X(3,-1)]+q[X(2)]\\
\mbox{}[X(3)]&\star&[X(3,1)]=-(t_1^2t_2+t_1^2t_3-t_2^3-t_2^2t_3)[X(3,1)]+(t_1^2+t_2^2+2t_2t_3)[X(3,2)]\\
&-&(t_1+t_3)q[X(3,-1)]+(t_1t_2+t_1t_3-t_2^2-t_2t_3)q[X(1)]-(t_1+t_2)q[X(2)]\\
&+&q[X(2,1)]+q[X(3)]\\
\mbox{}[X(3)]&\star&[X(3,2)]=-2t_2(t_1^2-t_3^2)[X(3,2)]+(t_1^2-t_3^2)q[X(3,-1)]\\
&+&2t_2(t_1-t_3)q[X(2)]-2t_2q[X(2,1)]-(t_1-t_3)q[X(3)]+q[X(3,1)]\\
\hline
\mbox{}[X(3,1)]&\star&[X(3,1)]= 2t_3 (t_1^2 - t_2^2)(t_2 + t_3)[X(3,1)]-2(t_1^2t_2+t_1^2t_3+t_2^2t_3+t_2t_3^2)[X(3,2)]\\
&+&(t_1^2+2t_1t_3+t_3^2)q[X(3,-1)]-2t_3(t_1t_2+t_1t_3-t_2^2-t_2t_3)q[X(1)]\\
&+&2t_1(t_2+t_3)q[X(2)]-2(t_2+t_3)q[X(2,1)]-(2t_1+t_2+t_3)q[X(3)]+q[X(3,1)]+q^2\\
\mbox{}[X(3,1)]&\star&[X(3,2)]=2t_2(t_1^2t_2+t_1^2t_3-t_2t_3^2-t_3^3)[X(3,2)]-(t_2+t_3)(t_1^2-t_3^2)q[X(3,-1)]\\
&-& 2t_2(t_1t_2+t_1t_3-t_2t_3-t_3^2)q[X(2)]+2t_2(t_2+t_3)q[X(2,1)]\\
&+&(t_1+2t_2+t_3)(t_1-t_3)q[X(3)]-(t_1+2t_2+t_3)q[X(3,1)]-(t_1-t_3)q^2+q^2[X(1)]
\end{eqnarray*}
\newpage
\begin{eqnarray*}
\mbox{}[X(3,2)]&\star&[X(3,2)]=-2t_2(t_1^2t_2^2-t_1^2t_3^2-t_2^2t_3^2+t_3^4)[X(3,2)]+(t_1^2t_2^2-t_1^2t_3^2-t_2^2t_3^2+t_3^4)q[X(3,-1)]\\
&+&2t_2(t_1t_2^2-t_1t_3^2-t_2^2t_3+t_3^3)q[X(2)]-2t_2(t_2^2-t_3^2)q[X(2,1)]\\
&-&(t_1+2t_2+t_3)(t_1t_2-t_1t_3-t_2t_3+t_3^2)q[X(3)]+(t_1^2+2t_1t_2+2t_2^2-t_3^2)q[X(3,1)]\\
&+&(t_1t_2-t_1t_3-t_2t_3+t_3^2)q^2-(t_1+t_2)q^2[X(1)]+q^2[X(2)]
\end{eqnarray*}
\end{small}

\def\cprime{$'$}
\providecommand{\bysame}{\leavevmode\hbox to3em{\hrulefill}\thinspace}
\providecommand{\MR}{\relax\ifhmode\unskip\space\fi MR }
\providecommand{\MRhref}[2]{%
  \href{http://www.ams.org/mathscinet-getitem?mr=#1}{#2}
}
\providecommand{\href}[2]{#2}


\end{document}